%% file: main.tex
\DeclareSymbolFont{AMSb}{U}{msb}{m}{n}
\documentclass[11pt,a4paper,twoside,leqno,noamsfonts]{amsart}
           \makeatletter
           \newcommand{\mylabel}[2]{#2\def\@currentlabel{#2}\label{#1}}
           \makeatother
\usepackage[english]{babel}
\usepackage[dvipsnames]{xcolor}
\usepackage{graphicx}
\usepackage[utopia]{mathdesign}
\usepackage[a4paper,top=3cm,bottom=2.8cm,left=2.8cm,
           right=2.8cm,bindingoffset=5mm,marginparwidth=70pt]{geometry}
\usepackage[utf8]{inputenc}
\usepackage{braket,caption,comment,mathtools,stmaryrd}
\usepackage[usestackEOL]{stackengine}
\usepackage{multirow,booktabs,microtype,relsize}
\usepackage[colorlinks,bookmarks,backref=page]{hyperref} %
      \hypersetup{colorlinks,%
            citecolor=britishracinggreen,%
            filecolor=black,%
            linkcolor=cobalt,%
            urlcolor=black}
      \numberwithin{equation}{section}
\input{macros.tex}

\title[The equivariant Atiyah class]{The equivariant Atiyah class}
\author[Andrea T. Ricolfi]{Andrea T. Ricolfi} 

\begin{document}
\maketitle

\begin{abstract}
Let $X$ be a complex scheme acted on by an affine algebraic group $G$. We prove that the Atiyah class 
of a $G$-equivariant perfect complex on $X$, as constructed by Huybrechts and Thomas, is $G$-equivariant in a precise sense. As an application, we show that, if $G$ is reductive, the obstruction theory on the fine relative moduli space $M\to B$ of simple  perfect complexes on a $G$-invariant smooth projective family $Y\to B$ is $G$-equivariant. The results contained here are meant to suggest how to check the equivariance of the natural obstruction theories on a wide variety of moduli spaces equipped with a torus action, arising for instance in Donaldson--Thomas theory and Vafa--Witten theory.
\end{abstract}

{\hypersetup{linkcolor=black}
\tableofcontents}

\section{Introduction}

\subsection*{Overview}
The \emph{Atiyah class} of a vector bundle $V$ on a complex algebraic variety $X$, introduced in \cite{MR86359}, is an extension class
\[
\At_V \,\in\, \Ext^1_X(V,V\otimes \Omega_X)
\]
whose vanishing is equivalent to the existence of an algebraic connection on $V$. 
A general definition of Atiyah class was given by Illusie for every complex of sheaves on a scheme~\cite{Ill}.

In this paper all schemes, stacks and group schemes are defined over $\BC$.
Let $G$ be an algebraic group acting on a scheme $X$. Our first goal is to make sense of, and prove (see Theorem \ref{main_thm}), a rigorous version of the following slogan:
\[
\emph{The Atiyah class of a }G\emph{-equivariant perfect complex on }X\emph{ is }G\emph{-equivariant.}
\]
Our main motivation comes from enumerative geometry: the Atiyah class is a crucial ingredient in the construction of the \emph{obstruction theory} \cite{BFinc,LiTian} on various moduli spaces of sheaves, such as those appearing as main characters in Donaldson--Thomas theory, Pandharipande--Thomas theory, Vafa--Witten theory.
When the moduli space is acted on by a torus $\BG_m^r$, a powerful tool to compute the virtual invariants defined via these obstruction theories is the \emph{virtual localisation formula}, proved in equivariant Chow cohomology by Graber and Pandharipande \cite{GPvirtual} and in K-theory by Fantechi and G\"{o}ttsche \cite{Fantechi_Gottsche}. The localisation theorem requires as input an \emph{equivariant} obstruction theory.
After confirming an equivariant version of Grothendieck duality (\S\,\ref{sec:verdier}), we show in a general example that the equivariance of the Atiyah class ensures that the obstruction theory it induces is itself equivariant, at least for reductive groups; since $\BG_m^r$ is reductive, this is enough to apply the virtual localisation formula.

By the above slogan, in a typical situation a `working mathematician' would only have to verify the equivariance of the universal sheaf\footnote{Most sheaf-theoretic moduli problems, including those for which one does not have a universal sheaf (but only a universal \emph{twisted} sheaf), should behave in a way that is entirely parallel to our discussion in \S\,\ref{sec:perfect_complexes}.} (or complex) on the moduli space under consideration in order to apply the localisation theorem. This procedure is explained in detail in \S\,\ref{sec:perfect_complexes}, where we prove the equivariance of the universal object on the moduli space of simple perfect complexes on a $G$-invariant smooth projective family (Proposition \ref{prop:At_E_is_equiv}). 
A special case of our construction will be considered in \cite{FMR} in the context of Quot schemes over $3$-folds, in order to produce equivariant virtual fundamental classes and prove a special case of a formula conjectured in \cite{Quot19}. Quot schemes have already appeared in many equivariant and non-equivariant calculations in enumerative geometry \cite{BR18,DavisonR,Gholampour2017a,Oprea:2019aa,Ricolfi2018,LocalDT}, and this work is meant to make the foundations of their virtual equivariant theory rigorous, as well as that of other moduli spaces of sheaves. 

\subsection*{Main result}
Let $X$ be a separated noetherian scheme over $\BC$, and let $\QCoh_X$ be the abelian category of quasi-coherent $\OO_X$-modules.
Let $E \in \Perf X \subset \derived(\QCoh_X)$ be a perfect complex. Assuming $X$ admits a closed embedding in a smooth scheme,
Huybrechts and Thomas defined the \emph{truncated Atiyah class} of $E$ as an element
\begin{equation}\label{main_atiyah}
\At_E \,\in\, \Ext^1_X(E,E\otimes \BL_X),
\end{equation}
where $\BL_X \in \derived^{[-1,0]}(\QCoh_X)$ is the truncated cotangent complex.
If $X$ carries an action of a complex algebraic group $G$, one can form the abelian category $\QCoh_X^G$ of $G$-equivariant quasi-coherent sheaves on $X$. There is an exact functor $\Phi\colon \derived(\QCoh_X^G)\to \derived(\QCoh_X)$ forgetting the equivariant structure.
We say that $\At_E$ is $G$-\emph{equivariant} if the corresponding morphism $E \to E\otimes\BL_X[1]$ in $\derived(\QCoh_X)$ can be lifted to $\derived(\QCoh_X^G)$ along $\Phi$.

\smallbreak
Our first main result is the following.

\begin{thm}[Theorem \ref{main_thm_BODY}]\label{main_thm}
Let $G$ be a complex affine algebraic group acting on a separated noetherian $\BC$-scheme $X$ admitting a $G$-equivariant embedding in a smooth $G$-scheme. 
Fix a perfect complex $E \in \Perf X$. Then every lift of $E$ to $\derived(\QCoh_X^G)$ makes $\At_E$ canonically $G$-equivariant. 
\end{thm}

Note that the assumption on $X$ is satisfied as soon as, for instance, $X$ is quasi-projective and carries at least one $G$-equivariant line bundle. 

\subsection*{Application to equivariant obstruction theories}

As we briefly recall below, the Atiyah class is the main ingredient in the construction of an \emph{obstruction theory} on moduli spaces of simple perfect complexes 
on a smooth projective family $Y\to B$, see \cite[Thm.~4.1]{HT}.
An obstruction theory \cite[Def.~4.4]{BFinc} on a scheme $X$ is a morphism $\phi\colon \mathbb E \to \BL_X$ 
in $\derived(\QCoh_X)$ such that $\mathcal H^0(\phi)$ is an isomorphism and $\mathcal H^{-1}(\phi)$ is a surjection. See \cite[\S\,7]{BFinc} for a relative version.

In case $X$ is acted on by an algebraic group $G$, the complex $\BL_X$ has a canonical lift to $\derived(\QCoh_X^G)$ (see \S\,\ref{sec:truncated_cotg}), and one has the following notion.

\begin{definition}[{\cite{GPvirtual,BFHilb}}]\label{def:equiv_OB_THEORY}
Let $G$ be an algebraic group acting on $X$. An obstruction theory $\phi\colon \mathbb E \to \BL_X$ is $G$-\emph{equivariant} if $\phi$ can be lifted to a morphism in $\derived(\QCoh_X^G)$.
\end{definition}

To make a statement about equivariant obstruction theories, one needs to get a handle on Hom-sets in $\derived(\QCoh_X^G)$. For this, we restrict to $G$ reductive (in order to exploit a technical result, Lemma \ref{lemma:G-inv_Ext_groups}). For instance, the theory works for a torus $G = \BG_m^r$, which includes most applications we have in mind.

\smallbreak
Here is the statement of our second main result.

\begin{thm}[Theorem \ref{main2_body}]\label{thm:main2}
Let $G$ be an affine reductive algebraic group. Let $Y\to B$ be a $G$-invariant smooth projective family of varieties. Let $M\to B$ be a fine moduli space of simple perfect complexes on the fibres of $Y\to B$. Then the relative obstruction theory on $M\to B$ is $G$-equivariant.
\end{thm}

We refer to \cite[\S\,4.1]{HT} (or our \S\,\ref{sec:moduli_M}) for the precise assumptions on $M\to B$.
We briefly outline here the role of the Atiyah class and of Grothendieck duality in the construction of the relative obstruction theory on $M\to B$. Set $X = Y\times_BM$, let $E \in \Perf X$ be the universal perfect complex and let $\pi_M\colon X \to M$ be the projection. The Atiyah class one has to consider is the `$M$-component' of \eqref{main_atiyah}, namely
\[
\At_{E/Y} \in \Ext^1_X(E,E\otimes \pi_M^\ast \BL_{M/B}).
\]
Via the distinguished triangle
\[
\RRlHom_X(E,E)_0 \to \RRlHom_X(E,E) \xrightarrow{\mathrm{trace}} \OO_X,
\]
the class $\At_{E/Y}$ projects onto an element of
\begin{equation}\label{vd}
\Ext^1_X(\RRlHom_X(E,E)_0,\pi_M^\ast \BL_{M/B})  \,\cong\, 
\Ext^{1-d}_M(\RR \pi_{M\ast}(\RRlHom_X(E,E)_0 \otimes \omega_{\pi_M}),\BL_{M/B}),
\end{equation}
where $d$ is the relative dimension of $Y\to B$ and the isomorphism is given by Grothendieck duality. The image of $\At_{E/Y}$ along this journey is a morphism
\[
\phi\colon \RR \pi_{M\ast}(\RRlHom_X(E,E)_0 \otimes \omega_{\pi_M})[d-1] \to \BL_{M/B},
\]
and it is shown in \cite[Thm.~4.1]{HT} that $\phi$ is a relative obstruction theory in the sense of \cite[\S\,7]{BFinc}. The following strategy will prove Theorem \ref{thm:main2}: 
\begin{enumerate}
\item  $X$ has a $G$-action such that $E$ is $G$-equivariant (Proposition \ref{prop:At_E_is_equiv}).
\item  $\At_{E/Y}$ is a $G$-invariant extension (thanks to Theorem \ref{main_thm}). \label{item2}
\item The Grothendieck duality isomorphism \eqref{vd} is $G$-equivariant (Corollary \ref{cor:Equivariant_Grothendieck_Global_sections}).\label{item3}
\item $G$-invariant extensions correspond to morphisms in $\derived(\QCoh_M^G)$.\label{item4}
\end{enumerate}

Step \eqref{item4} uses the reductivity of $G$ via Lemma \ref{lemma:G-inv_Ext_groups}.

\begin{conventions*}
All schemes are noetherian and defined over $\BC$. By an \emph{algebraic group} $G$ we mean a connected 
group scheme of finite type over $\BC$ (often affine). We follow Olsson \cite[Ch.~8]{Olsson_book} for conventions on algebraic stacks (in particular, we make no separation assumptions).
For an algebraic stack $\mathscr X$, we denote by $\QCoh_{\mathscr X}$ the abelian category of quasi-coherent 
sheaves on the lisse-\'etale site of $\mathscr X$ \cite[Ch.~9]{Olsson_book}, and $\derived(\mathscr X)$ will denote the unbounded derived category of the abelian category $\Mod_{\OO_{\mathscr X}}$ of all $\OO_{\mathscr X}$-modules. 
\end{conventions*}

\subsection*{Acknowledgements}
We thank Pieter Belmans, Martijn Kool, Amnon Neeman and Richard Thomas for very helpful discussions. Special thanks to David Rydh for suggesting several improvements. We also thank SISSA for the excellent working conditions.

\section{Equivariant sheaves and complexes}

\subsection{The category of equivariant sheaves}\label{sec:equiv_sheaves_preliminaries}

Let $X$ be a noetherian scheme over $\BC$,\footnote{The theory works relatively to a fixed base scheme $B$ (see Remark \ref{rmk:relative}). This requires all relative operations (such as fibre products) to be performed over $B$, as well as the requirement that $G \to B$ be flat (this would be needed e.g.~in the construction of $f_\ast$ and $\lHom_X$ in \S\,\ref{sec:underived_functors}).} equipped with an action $\sigma\colon G\times X \to X$ of a group scheme $G$. We call such a pair $(X,\sigma)$ a $G$-\emph{scheme}. The abelian category $\Mod_{\OO_X}$ of $\OO_X$-modules contains the abelian subcategories $\QCoh_X$ (resp.~$\Coh_X$) of quasi-coherent (resp.~coherent) $\OO_X$-modules. We will mostly focus on $\QCoh_X$ in this paper.

Denoting by $\mathsf m\colon G\times G\to G$ the group law of $G$, there is a commutative diagram
\begin{equation}\label{diag:group_law}
\begin{tikzcd}[column sep = large]
G\times G\times X\arrow[swap]{d}{\id_G\times\sigma}\arrow{r}{\mathsf m\times \id_X} &
G\times X\arrow{d}{\sigma} \\
G\times X \arrow{r}{\sigma} & 
X
\end{tikzcd}
\end{equation}
translating the condition $g\cdot (h\cdot x) = (gh)\cdot x$.

Let $p_i\colon G\times X \to X$ and $p_{ij}\colon G\times G\times X \to G\times X$ denote the projections onto the labeled factors.

\begin{definition}\label{def:equivariant_sheaves}
A $G$-\emph{equivariant quasi-coherent sheaf} on $X$ is a pair $(\mathscr F,\vartheta)$ where $\mathscr F \in \QCoh_X$ and $\vartheta\colon p_2^\ast \mathscr F \simto \sigma^\ast\mathscr F$ is an isomorphism in $\QCoh_{G\times X}$ compatible with the diagram \eqref{diag:group_law}. In other words, $\vartheta$ is required to satisfy the cocycle condition 
\begin{equation}\label{eqn:cocycle_condition}
(\mathsf m\times \id_X)^\ast \vartheta = (\id_G\times\sigma)^\ast\vartheta \circ p_{23}^\ast \vartheta.
\end{equation}
The isomorphism $\vartheta$ is called a $G$-\emph{equivariant structure} on $\mathscr F$.
\end{definition}

The same definition can be given for objects $\mathscr F \in  \Mod_{\OO_X}$ as well as $\mathscr F \in \Coh_X$.

\smallbreak
Explicitly, the cocycle condition \eqref{eqn:cocycle_condition} means that the diagram of isomorphisms
\[
\begin{tikzcd}[column sep = large]
(\mathsf m\times \id_X)^\ast p_2^\ast \mathscr F \arrow{rr}{(\mathsf m\times \id_X)^\ast \vartheta}\arrow[equal]{d} 
& & (\mathsf m\times \id_X)^\ast  \sigma^\ast\mathscr F\arrow[equal]{d} \\
p_{23}^\ast p_2^\ast \mathscr F\arrow{d}{p_{23}^\ast\vartheta}
& & (\id_G\times\sigma)^\ast  \sigma^\ast\mathscr F \\
p_{23}^\ast\sigma^\ast\mathscr F \arrow[equal]{rr} 
& & (\id_G\times\sigma)^\ast p_2^\ast \mathscr F \arrow[swap]{u}{(\id_G\times\sigma)^\ast\vartheta}
\end{tikzcd}
\]
commutes in $\QCoh_{G\times G\times X}$.

\begin{remark}\label{remark:action_is_FLAT}
Let $(X,\sigma)$ be a $G$-scheme. Then $\sigma$ is flat. Indeed, it agrees with the composition
\[
G \times X \xrightarrow{\gamma} G \times X \xrightarrow{p_2} X
\]
where $\gamma$ is the automorphism $(g,x) \mapsto (g,\sigma(g,x))$, having $(g,x) \mapsto (g,\sigma(g^{-1},x))$ as inverse.
\end{remark}

\begin{definition}
A morphism $(\mathscr F,\vartheta) \to (\mathscr F',\vartheta')$ of $G$-equivariant quasi-coherent sheaves is a morphism $\phi\colon \mathscr F \to \mathscr F'$ in $\QCoh_X$ such that the diagram
\begin{equation}\label{diag:morphism_of_equivariant_sheaves}
\begin{tikzcd}
p_2^\ast\mathscr F \arrow[swap]{d}{\vartheta} \arrow{r}{p_2^\ast \phi} &
p_2^\ast\mathscr F' \arrow{d}{\vartheta'} \\
\sigma^\ast \mathscr F \arrow{r}{\sigma^\ast\phi} &
\sigma^\ast \mathscr F'
\end{tikzcd}
\end{equation}
commutes in $\QCoh_{G\times X}$.
\end{definition}

\begin{notation}
Let $\mathcal C$ be any of the categories $\Mod_{\OO_X}$, $\QCoh_X$ or $\Coh_X$. We let $\mathcal C^G$ denote the corresponding category of $G$-equivariant sheaves $(\mathscr F,\vartheta)$ where $\mathscr F \in \mathcal C$. We mainly focus on $\mathcal C = \QCoh_X$.
The category $\QCoh_X^G$ is a $\BC$-linear abelian category 
(see also Lemma \ref{lemma:Grothendieck_cat} for a stronger statement).
Its (unbounded) derived category will be denoted $\derived(\QCoh_X^G)$. 
\end{notation}
 
Consider the composite isomorphism
\[
\rho_g\colon X \,\simto\, \set{g}\times X \into G\times X \xrightarrow{\sigma} X,\quad x \mapsto \sigma(g,x).
\]  
Every object $(\mathscr F,\vartheta) \in \QCoh_X^G$ comes with a collection of isomorphisms
\[
\vartheta_{g} \colon \mathscr F \,\simto\, \rho_g^\ast\mathscr F,\quad g\in G,
\]
satisfying $\vartheta_{hg} = \rho_g^\ast\vartheta_{h}\circ \vartheta_{g}$, where $\vartheta_{g}$ is the restriction of $\vartheta$ along $X\,\simto\,\set{g} \times X \subset G\times X$.

\begin{remark}[Relative version]\label{rmk:relative}
In general, when working with a flat group scheme $G\to B$ acting on a scheme $X\to B$, where $B$ is a base scheme, 
a $G$-equivariant sheaf $(\mathscr F,\vartheta)$ can be described in the following equivalent fashion. 
Some notation first. For every $B$-scheme $T$, set $X_T = T\times_BX = T\times_T X_T$ and 
let $\mathscr F_T$ denote the pullback of $\mathscr F$ along the projection $X_T \to X$. For 
every $T$-valued point $g\colon T \to G_T = T\times_BG$ of $G$ one has an isomorphism of $T$-schemes
\[
\rho_g\colon X_T \xrightarrow{g \times \id_{X_T}} G_T\times_T X_T \xrightarrow{\sigma_T}X_T,\quad (t,x) \mapsto (t,\sigma_T(g(t),x)).
\]
The condition `$\mathscr F$ is $G$-equivariant' is equivalent to the following condition: for every $T$-valued point $g \in G_T(T)$ as above there is an isomorphism $\vartheta_g \colon \mathscr F_T \simto \rho_g^\ast \mathscr F_T$ 
such that for every pair of $T$-valued points $g$, $h \in G_T(T)$ one has a commutative diagram of isomorphisms
\begin{equation}\label{equivariance_diagram}
\begin{tikzcd}
\rho_{g}^\ast \rho_{h}^\ast \mathscr F_T \arrow[equal]{d} & \rho_g^\ast \mathscr F_T \arrow[swap]{l}{\rho_g^\ast \vartheta_h} \\
\rho_{hg}^\ast \mathscr F_T & \mathscr F_T\arrow[swap]{l}{\vartheta_{hg}}\arrow[swap]{u}{\vartheta_g} 
\end{tikzcd}
\end{equation}
in $\QCoh_{X_T}$.
\end{remark}

\begin{example}\label{ex:equivariant_structure_sheaf}
Let $(X,\sigma)$ be a $G$-scheme over a scheme $B$. Then the structure sheaf $\OO_X$ is $G$-equivariant in a natural way. 
For a $B$-scheme $T$, set $X_T = T\times_BX$. Then the inverse of the natural isomorphisms 
$\rho_g^\ast \OO_{X_T} \simto \rho_g^\ast \rho_{g\ast}\OO_{X_T} \simto \OO_{X_T}$ is a $G$-equivariant structure on $\OO_X$.
\end{example}

\begin{example}
Let $(X,\sigma)$ be a $G$-scheme over a scheme $B$. Then the sheaf $\Omega_{X/B}$ of relative differentials is $G$-equivariant in a natural way. 
Indeed, for a $B$-scheme $T$, consider the natural isomorphisms $\alpha_T\colon (\Omega_{X/B})_T\simto \Omega_{X_T/T}$ and $\ell_g\colon \rho_g^\ast \Omega_{X_T/T}\simto \Omega_{X_T/T}$, where $g \in G_T(T)$. Then the composition
\[
\vartheta_g\colon (\Omega_{X/B})_T\xrightarrow{\alpha_T} \Omega_{X_T/T} \xrightarrow{\ell_g^{-1}} \rho_g^\ast \Omega_{X_T/T}\xrightarrow{\rho_g^\ast \alpha_T^{-1}} \rho_g^\ast (\Omega_{X/B})_T
\]
defines an equivariant structure on $\Omega_{X/B}$.
\end{example}

\begin{notation}
For an object $(\mathscr F,\vartheta) \in \QCoh_X^G$, we will often somewhat sloppily omit the $G$-equivariant structure `$\vartheta$' from the notation.
We will also write $\Hom_X$ instead of $\Hom_{\QCoh_X}$ or $\Hom_{\derived(\QCoh_X)}$, and write $g^\ast$ instead of $\rho_g^\ast$.
\end{notation}

\begin{remark}\label{remark:Hom_is_a_G-rep}
If $(\mathscr F,\vartheta_{\mathscr F})$, $(\mathscr F',\vartheta_{\mathscr F'}) \in \QCoh_X^G$, the $\BC$-vector space
\[
\Hom_{X}(\mathscr F,\mathscr F')
\]
is naturally a $G$-representation. Indeed, for a morphism
$\phi\colon \mathscr F \to \mathscr F'$ in $\QCoh_X$, one defines $g\cdot \phi$ by means of the composition
\begin{equation}\label{diag:action_on_Homs}
\begin{tikzcd}[row sep=large,column sep=large]
\mathscr F \arrow[swap]{d}{\vartheta_{\mathscr F,g}} \arrow[dotted]{r}{g\cdot \phi} & \mathscr F' \\
g^\ast\mathscr F \arrow{r}{g^\ast\phi} & g^\ast\mathscr F'\arrow[swap]{u}{\vartheta^{-1}_{\mathscr F',g}}
\end{tikzcd}
\end{equation}
exploiting the invertibility of $\vartheta_{\mathscr F',g}$. The structure of $G$-representation 
on $\Hom_{X}(\mathscr F,\mathscr F')$ clearly depends on the chosen equivariant structures $\vartheta_{\mathscr F}$ and $\vartheta_{\mathscr F'}$.
\end{remark}

\begin{remark}\label{remark:G-inv_Homs}
It is immediate to see that, in $\QCoh_X^G$, the morphisms are the $G$-invariant morphisms between the underlying quasi-coherent sheaves. In symbols,
\begin{equation}\label{G-inv_homs}
\Hom_{\QCoh_X^G}((\mathscr F,\vartheta_{\mathscr F}),(\mathscr F',\vartheta_{\mathscr F'})) = \Hom_{X}(\mathscr F,\mathscr F')^G.
\end{equation}
Indeed, the diagram \eqref{diag:morphism_of_equivariant_sheaves} becomes precisely 
\[
\begin{tikzcd}
\mathscr F \arrow[swap]{d}{\vartheta_{\mathscr F,g}} \arrow{r}{\phi} & \mathscr F'\arrow{d}{\vartheta_{\mathscr F',g}} \\
g^\ast\mathscr F \arrow{r}{g^\ast\phi} & g^\ast\mathscr F'
\end{tikzcd}
\]
when restricted to $\set{g} \times X\cong X$. Again, in the right hand side of \eqref{G-inv_homs} the `$G$-invariant part' depends on the $G$-structure on $\Hom_X(\mathscr F,\mathscr F')$, which in turn is determined by the pair $(\vartheta_{\mathscr F},\vartheta_{\mathscr F'})$.
\end{remark}

The following result is classical, and is key to this paper. It is proved in \cite[Ex.~12.4.6]{LMB}, but see also 
\cite[Exercise 9.H]{Olsson_book}.

\begin{prop}\label{theorem:derived(quotient_stack)}
Let $G$ be a smooth group scheme, $X$ a $G$-scheme. There is an equivalence
\[
\QCoh_{X}^{G} \,\cong \,\QCoh_{[X/G]}.
\]
\end{prop}

\begin{example}
Let $X = \Spec k$, for a field $k$. Then $\QCoh^{G}_{\Spec k} \cong \QCoh_{\mathrm B_kG}$, which in turn is equivalent to the category $\Rep_{k}(G)$ of locally finite $k$-linear representations of $G$.
\end{example}

\subsection{Forgetful functor}
There is an exact functor $\Phi\colon \derived(\QCoh_X^G) \to \derived(\QCoh_X)$ that forgets the equivariant structure. This results in a commutative diagram
\[
\begin{tikzcd}
\QCoh_X^G\arrow[hook]{d}\arrow{r}{\Phi} &
\QCoh_X \arrow[hook]{d} 
\\
\derived(\QCoh_X^G)\arrow{r}{\Phi} &
\derived(\QCoh_X)
\end{tikzcd}
\]
where the vertical arrows are the inclusions of the standard hearts. We write $\Phi_X$ when we wish to emphasise the scheme.
More concretely, if $\mathsf p\colon X \to [X/G]$ is the standard smooth atlas, we can identify $\Phi$ as the composition
\begin{equation}\label{factorisation_Phi}
\derived(\QCoh_X^G) \,\simto\, \derived(\QCoh_{[X/G]})  \xrightarrow{\mathsf p^\ast} \derived(\QCoh_X),
\end{equation}
where $\mathsf p^\ast = \LL \mathsf p^\ast$ is the pullback functor as defined in \cite[\S\,7]{Olsson} and the first equivalence comes from Proposition \ref{theorem:derived(quotient_stack)}.

\begin{remark}\label{remark:forget_reflects_exactness}
The forgetful functor $\Phi$ reflects exactness: a sequence in $\QCoh_X^G$ that becomes exact in $\QCoh_X$ was already exact in $\QCoh_X^G$. We will not need this fact.
\end{remark}

\subsection{Geometric functors}\label{sec:underived_functors}
Fix two noetherian $G$-schemes $(X,\sigma_X)$ and $(Y,\sigma_Y)$. All morphims $X \to Y$ in this subsection are assumed to be $G$-equivariant. 
Since $X$ and $Y$ are noetherian, pushforward preserves quasi-coherence.

Let $(\mathscr F,\vartheta)$ and $(\mathscr F',\vartheta')$ be two objects of $\QCoh_X^G$. Then there is a canonical lift
\[
(\mathscr F\otimes \mathscr F',\vartheta\otimes\vartheta') \in \QCoh_X^G
\]
of the object $\mathscr F\otimes \mathscr F' \in \QCoh_X$.
This gives a bi-functor
\[
\QCoh_X^G \times \QCoh_X^G \xrightarrow{\otimes} \QCoh_X^G.
\]
If $f\colon X \to Y$ is a morphism of $G$-schemes, there is a pullback functor
\[
\QCoh_Y^G \xrightarrow{f^\ast} \QCoh_X^G,\quad (\mathscr E,\vartheta_{\mathscr E}) \mapsto (f^\ast \mathscr E,(\id_G\times f)^\ast \vartheta_{\mathscr E}).
\]
By flat base change along $p_2,\sigma_Y\colon G\times Y \rightrightarrows Y$ (cf.~Remark \ref{remark:action_is_FLAT} for flatness of $\sigma_Y$), one constructs a pushforward functor
\[
\QCoh_X^G \xrightarrow{f_\ast} \QCoh_Y^G,\quad (\mathscr F,\vartheta_{\mathscr F}) \mapsto (f_\ast \mathscr F,(\id_G\times f)_\ast\vartheta_{\mathscr F}),
\]
such that $(f^\ast,f_\ast)$ is an adjoint pair. 

Finally, given $(\mathscr F,\vartheta) \in \Coh_X^G$ and $(\mathscr F',\vartheta')\in \QCoh_X^G$, there is a canonical $G$-equivariant structure on the quasi-coherent sheaf $\lHom_{X}(\mathscr F,\mathscr F')$,\footnote{Note that $\lHom_{X}(\mathscr F,\mathscr F')$ is quasi-coherent because we assumed $\mathscr F$ is coherent.} given by $\lHom_{G\times X}(\vartheta,\vartheta')$. More precisely, the isomorphism
\[
\lHom_{G\times X}(\vartheta,\vartheta')\colon \lHom_{G\times X}(p_2^\ast \mathscr F,p_2^\ast \mathscr F') \,\simto\,\lHom_{G\times X}(\sigma_X^\ast \mathscr F,\sigma_X^\ast \mathscr F')
\]
can be used to define the equivariant structure
\[
\begin{tikzcd}[row sep = large]
p_2^\ast \lHom_X(\mathscr F,\mathscr F') \isoarrow{d} \arrow[dotted]{rr}{\vartheta_{\lHom_X(\mathscr F,\mathscr F')}} & & \sigma_X^\ast\lHom_X(\mathscr F,\mathscr F') \\
\lHom_{G\times X}(p_2^\ast\mathscr F,p_2^\ast\mathscr F') \arrow{rr}{\lHom_{G\times X}(\vartheta,\vartheta')} & & \lHom_{G\times X}(\sigma_X^\ast \mathscr F,\sigma_X^\ast \mathscr F')\isoarrow{u}
\end{tikzcd}
\]
in the top row, where to obtain the vertical isomorphisms one exploits the flatness of $p_2$ and of $\sigma_X$, as well as the coherence of $\mathscr F$, see \cite[(6.7.6)]{MR0217083}. This construction defines a bi-functor
\[
\Coh_X^G \times \QCoh_X^G \xrightarrow{\lHom_{X}(-,-)} \QCoh_X^G.
\]

\section{Equivariant derived functors and Grothendieck duality}\label{sec:Equivariant_derived_functors}

This section contains the technical material needed to prove Theorems \ref{main_thm} and \ref{thm:main2}. It can be skipped at a first reading.

Throughout this section we fix an affine (connected) algebraic group $G$ over $\BC$. In particular $G$ is smooth by Cartier's theorem. Moreover, $G$ is \emph{linear}. 
We also assume (with the exception of \S\,\ref{sec:quasi-coh_stacks}, which just records some general definitions) that all schemes are noetherian and separated over $\BC$. 

\subsection{Ample families of equivariant line bundles}
Given a $G$-scheme $X$, the following condition will be crucial:
\[
(\dagger)\quad 
\begin{array}{c}
X\textrm{ has an ample } 
\textrm{family of }G\textrm{-equivariant line bundles}.
\end{array}
\]

Condition $(\dagger)$ means that there exists a family $\set{\mathscr L_i}_{i \in I}$ of $G$-equivariant line bundles such that, for every object $\mathscr E \in \QCoh_X$, the evaluation map yields a surjective morphism
\[
\bigoplus_i\bigoplus_{n>0} \HH^0(X,\mathscr E\otimes \mathscr L_i^{\otimes n}) \otimes_{\BC} (\mathscr L_i^\vee)^{\otimes n} \onto \mathscr E.
\]
Here the index set $I$ is arbitrary, but since $X$ is quasi-compact taking $I$ to be finite yields an equivalent definition.

\begin{example}\label{rmk:G-equiv_emb}
If $X$ is a quasi-projective scheme with a linear $G$-action, then $(\dagger)$ holds. If $X$ is quasi-projective and there exists at least one $G$-equivariant line bundle on $X$, then $X$ admits a $G$-equivariant embedding in a smooth scheme.
\end{example}

\begin{lemma}\label{lemma:equivariant_embedding_implies_dagger}
If $X$ admits a $G$-equivariant immersion in a smooth separated $\BC$-scheme, then $X$ satisfies $(\dagger)$.
\end{lemma}

\begin{proof}
Let $i\colon X \into A$ be a $G$-equivariant locally closed embedding, where $A$ is a smooth separated $\BC$-scheme. By \cite[II.2.2.7.1]{SGA6}, $A$ has an ample family of line bundles $\set{\mathscr L_i}_i$. By a result of Sumihiro \cite[Thm.~1.6]{Sumihiro1975}, every line bundle $\mathscr L$ on a normal scheme, such as $A$, has a tensor power $\mathscr L^{\otimes s}$ that is $G$-linearisable (this also uses smoothness of $G$). So $A$ has an ample family of $G$-equivariant line bundles. Since $i$ is a quasi-affine morphism (and is $G$-equivariant), pulling back this family along $i$ yields an ample family of $G$-equivariant line bundles.
\end{proof}

\begin{remark}\label{rmk:resol_prop}
By a result of Thomason \cite{Thomason}, we have the implication
\[
X \textrm{ satisfies }(\dagger)\,\,\,\Rightarrow\,\,\,[X/G]\textrm{ has the resolution property},
\]
where the condition on the right means that every $G$-equivariant \emph{coherent} $\OO_X$-module is the quotient of a $G$-equivariant locally free $\OO_X$-module of finite type.
\end{remark}

\begin{remark}
For a $G$-scheme $X$, Condition $(\dagger)$ is \emph{not} equivalent to the resolution property for the stack $[X/G]$.
In \cite[\S\,9]{Totaro} an example is given of a projective variety $X$ (a nodal cubic curve) acted on by an algebraic group $G$ (the torus $\mathbb G_m$), such that $X$ does not admit a family of $G$-equivariant line bundles. However, the quotient stack $[X/G]$ does have the resolution property \cite[Prop.~9.1]{Totaro}. 
\end{remark}

\subsection{Quasi-coherent sheaves on quotient stacks}\label{sec:quasi-coh_stacks}

For the sake of completeness, and for future reference, we record here a few properties of (quotient) stacks and their derived categories.

\subsubsection{Perfect complexes on schemes} \label{sec:perfect_complexes_on_schemes}

Let $X$ be an arbitrary scheme.

\begin{definition}[{\cite[Section~2]{MR1106918}}]\label{def:strictly_perfect}
 A complex $E\in \derived(X)$ is called \emph{perfect} (resp.~\emph{strictly perfect}) if it is locally (resp.~globally)
quasi-isomorphic to a bounded complex of locally free $\OO_X$-modules of finite type. We let $\Perf X$ denote the category of perfect complexes on $X$.
\end{definition}

\begin{remark}\label{perfect_and_strictly_perfect}
As long as $X$ is quasi-compact, quasi-separated and has an ample family of line bundles, there is no difference between perfect and strictly perfect \cite[Prop.~2.3.1\,(d)]{MR1106918}. For $G$-schemes satisfying condition $(\dagger)$, every perfect complex is then a \emph{bounded} complex. 
\end{remark}

By \cite[\href{https://stacks.math.columbia.edu/tag/08DB}{Tag 08DB}]{stacks-project}, if $X$ is quasi-compact and semi-separated (i.e.~has affine diagonal), the canonical functor $\derived(\QCoh_X) \to \derived_{\qcoh}(X)$ is an equivalence. The same holds true for any noetherian scheme \cite[\href{https://stacks.math.columbia.edu/tag/09TN}{Tag 09TN}]{stacks-project}. Here the decoration `$\qcoh$' means that the cohomology sheaves of the complexes lie in $\QCoh_{X}$. Our schemes will be noetherian, so all statements usually made about $\derived_{\qcoh}(X)$ can, and will be rephrased here using $\derived(\QCoh_X)$.

\subsubsection{Separation and noetherianity for algebraic stacks}\label{sec:quasi-sep}
Let $S$ be a scheme. Recall that a morphism of schemes $X \to S$ is \emph{quasi-separated} if the diagonal $X\to X \times_SX$ is quasi-compact. On the other hand, an algebraic stack $\mathscr X \to S$ is \emph{quasi-separated} if the diagonal $\mathscr X \to \mathscr X\times_S\mathscr X$ is quasi-compact \emph{and quasi-separated}, see~\cite[Def.~8.2.12]{Olsson_book}.

Let $G\to S$ be a smooth quasi-compact separated group scheme acting on a quasi-compact quasi-separated $S$-scheme $X\to S$. An algebraic stack of the form
$[X/G] \to S$ has representable, quasi-compact and separated diagonal~\cite[Ex.~4.6.1]{LMB}, therefore it is quasi-separated. 

An algebraic stack is \emph{noetherian} if it is quasi-compact, quasi-separated and admits a noetherian atlas. 
For instance, if $X$ is a noetherian scheme acted on by a smooth affine algebraic group, then $[X/G]$ is a noetherian algebraic stack. Indeed, $X \to [X/G]$ is an atlas; we just established quasi-separatedness, and quasi-compactness can be checked on an atlas \cite[\href{https://stacks.math.columbia.edu/tag/04YA}{Tag 04YA}]{stacks-project}.

\subsubsection{Compact generation for derived categories}
For an algebraic stack $\mathscr X$, the inclusion $\QCoh_{\mathscr X} \subset \Mod_{\OO_{\mathscr X}}$ of the abelian category of quasi-coherent $\OO_{\mathscr X}$-modules on the lisse-\'etale site of $\mathscr X$ (cf.~\cite[Def.~9.1.6]{Olsson_book}) inside the abelian category of all $\OO_{\mathscr X}$-modules induces a canonical functor
\[
\derived(\QCoh_{\mathscr X}) \to \derived_{\qcoh}(\mathscr X) \subset \derived(\mathscr X).
\]

We now briefly recall the notion of compact generation. It will be essential in the proof of equivariant Grothendieck duality (Theorem \ref{thm:equivariant_Grothendieck}).

\begin{definition}[{\cite[Def.~1.7]{Neeman2}}]\label{def:compact_generation}
A triangulated category $\mathcal S$ with small coproducts is said to be \emph{compactly generated} if there is a set of objects $\mathcal S' \subset \mathcal S$ such that for every $s \in \mathcal S'$ the functor $\Hom_{\mathcal S}(s,-)$ commutes with coproducts, and whenever $y$ is an object of $\mathcal S$ such that $\Hom_{\mathcal S}(s,y)=0$ for all $s \in \mathcal S'$, then it follows that $y=0$.
\end{definition}

\begin{prop}\label{prop:compact_generation}
Let $X$ be a noetherian scheme over $\BC$, acted on by an affine algebraic group $G$ and satisfying $(\dagger)$. Set $\mathscr X = [X/G]$.
The derived category $\derived(\QCoh_{\mathscr X})$ has small coproducts and is compactly generated.
\end{prop}

\begin{proof}
First of all, $\derived(\QCoh_{\mathscr X})$ has small coproducts because $\QCoh_{\mathscr X}$ is a Grothendieck abelian category \cite[\href{https://stacks.math.columbia.edu/tag/06WU}{Tag 06WU}]{stacks-project}.

Next, $\mathscr X$ is quasi-separated and noetherian: this was established in \S\,\ref{sec:quasi-sep}.
Since $G$ is affine, $\mathscr X$ has affine stabiliser groups at closed points. By \cite[Prop.~1.3]{Totaro}, a noetherian algebraic stack with affine stabiliser groups at closed points, and having the resolution property, has affine diagonal. Therefore, by Remark \ref{rmk:resol_prop}, $\mathscr X$ has affine diagonal.

Since we work in characteristic $0$ and $X$ has an ample family $\set{\mathscr L_i}$ of $G$-equivariant line bundles, $\mathscr X$ is concentrated (see \cite[Ex.~8.6]{Hall_Rydh_1}, and \cite[Def.~2.4]{Hall_Rydh_1} for the definition of concentrated), thus it has the compact resolution property; therefore we can apply  \cite[Prop.~8.4]{Hall_Rydh_1} to conclude that the family of line bundles $\set{\mathscr L_i}$ (including their shifts), viewed as line bundles over $\mathscr X$, form a family of compact generators for $\derived_{\qcoh}(\mathscr X)$. 

Summing up, $\mathscr X$ is a quasi-compact algebraic stack with affine diagonal and such that $\derived_{\qcoh}(\mathscr X)$ is compactly generated: by \cite[Thm.~1.2]{Hall_Neeman_Rydh} this implies that the canonical functor $\derived(\QCoh_{\mathscr X}) \to \derived_{\qcoh}(\mathscr X)$ is an equivalence. Thus $\derived(\QCoh_{\mathscr X})$ is compactly generated.
\end{proof}

\begin{remark}
See also \cite[Rem.~8.7]{Hall_Rydh_1} for the statement (an example of application of \cite[Prop.~8.4]{Hall_Rydh_1}) that for a quasi-compact concentrated algebraic stack $\mathscr X$ having affine diagonal and the resolution property, the category $\derived_{\qcoh}(\mathscr X)$ is compactly generated.
\end{remark}

\begin{corollary}\label{cor:derived_X^GcompactlyGenerated}
The category $\derived(\QCoh_{X}^{G})$ is compactly generated.
\end{corollary}

\begin{proof}
This follows directly from Proposition \ref{theorem:derived(quotient_stack)}.
\end{proof}

\begin{remark}
Combining Proposition \ref{theorem:derived(quotient_stack)} with the proof of Proposition \ref{prop:compact_generation} shows that if $X$ is a noetherian $\BC$-scheme acted on by an affine algebraic group $G$, and satisfying $(\dagger)$, then we have equivalences
\begin{equation}\label{equivalences}
\derived(\QCoh_X^G) \,\simto\,\derived(\QCoh_{[X/G]}) \,\simto\, \derived_{\qcoh}([X/G]).
\end{equation}
The literature on derived functors for algebraic stacks usually refers to $\derived_{\qcoh}$, but given the equivalences \eqref{equivalences} implied by our assumptions, we will state the results we need for $\derived(\QCoh^G)$.
\end{remark}

\begin{example}
Let $G$ be a group scheme of finite type over a field $k$ of characteristic $0$. Then $\derived_{\qcoh}(\mathrm{B}_{k}G)$ is compactly generated. Moreover, if $G$ is affine, it is compactly generated by the irreducible $k$-representations of $G$, see~\cite[Thm.~A]{MR3436239}. Let $\Rep_{k}(G)$ be the abelian category of $k$-linear locally finite  representations of $G$. Then by \cite[Thm.~1.2]{Hall_Neeman_Rydh} the natural functor 
$\derived(\Rep_{k}(G)) = \derived(\QCoh_{\mathrm{B}_{k}G}) \to \derived_{\qcoh}(\mathrm{B}_{k}G)$ is an equivalence. 
\end{example}

\subsection{Equivariant derived functors}\label{sec:equiv_der_functors}
In this subsection $G$ denotes an affine (connected) complex algebraic group. All $G$-schemes are noetherian, separated over $\BC$, and satisfy $(\dagger)$. In particular all morphisms, which we always assume to be $G$-equivariant, are quasi-compact and separated.

The following result lies at the foundations of the construction of derived versions of the geometric functors recalled in \S\,\ref{sec:underived_functors}.

\begin{lemma}[{\cite[Prop.~1.5.7\,(a) and Prop.~1.5.6~(a)]{VV}}]\label{lemma:Grothendieck_cat}
Let $X$ be a $G$-scheme. The category $\QCoh_X^G$ is a Grothendieck abelian category with enough injectives. 
Moreover, any complex of objects in $\QCoh_X^G$ has a K-injective resolution and a K-flat resolution.
\end{lemma}

\begin{remark}
Serp\'e proved that, in fact, unbounded complexes on any Grothendieck category admit K-injective resolutions \cite[Thm.~3.13]{MR1948842}.
\end{remark}

\begin{remark}
The resolutions mentioned in Lemma \ref{lemma:Grothendieck_cat} (which are carefully defined in \cite[Def.~1.5.3\,(c),\,(d)]{VV}) are precisely the $G$-equivariant analogues of those used by Spaltenstein to construct derived functors for unbounded derived categories in the non-equivariant case, see in particular Definitions 1.1 and 5.1 in \cite{Spaltenstein}.
\end{remark}

K-flat and K-injective resolutions allow one to define equivariant derived functors. As explained in \cite[\S\,1.5]{VV}, K-flat resolutions are needed to construct derived tensor product and derived pullback, whereas K-injective resolutions are used to construct derived pushforward.

\begin{prop}[{\cite[Prop.~1.5.6,~1.5.7]{VV}}]
Let $X$ be a $G$-scheme. There is a left derived functor
\[
\otimes^{\LL}\colon \derived(\QCoh_X^G) \times \derived(\QCoh_X^G) \to \derived(\QCoh_X^G).
\]
If $f\colon X\to Y$ is a 
morphism of $G$-schemes, there is a left derived functor
\[
\LL f^\ast\colon \derived(\QCoh_Y^G) \to \derived(\QCoh_X^G),
\]
and a right derived functor
\[
\RR f_\ast\colon \derived(\QCoh_X^G) \to \derived(\QCoh_Y^G).
\]
\end{prop}

To construct the equivariant derived sheaf Hom functor, we proceed as follows. We consider the quotient stack $\mathscr X = [X/G]$ and the category $\Perf (\mathscr X) \subset \derived_{\qcoh}(\mathscr X)$ of perfect complexes. We also let $\Perf^G(X) \subset \derived(\QCoh_X^G)$ be the image of $\Perf(\mathscr X)$ under the equivalence $\derived_{\qcoh}(\mathscr X)\simto \derived(\QCoh_X^G)$.

By \cite[Lemma 4.3\,(2)]{Hall_Rydh_1}, the bifunctor 
\begin{equation}\label{derived_rhom}
\RRlHom_{\mathscr X}(-,-)\colon \derived(\mathscr X) \times \derived(\mathscr X)^{\op} \to \derived(\mathscr X) 
\end{equation}
calculated in $\OO_{\mathscr X}$-modules restricts to
\[
\RRlHom_{\mathscr X}(-,-)\colon \Perf(\mathscr X) \times \derived(\QCoh_{\mathscr X})^{\op}\to \derived(\QCoh_{\mathscr X}).
\]
Let $\mathsf p\colon X \to \mathscr X$ denote the standard atlas. Exploiting the equivalences \eqref{equivalences} and the factorisation \eqref{factorisation_Phi}, we obtain a diagram
\[
\begin{tikzcd}[row sep = large]
\Perf(\mathscr X) \times \derived(\QCoh_{\mathscr X})^{\op} \isoarrow{d}\arrow{rr}{\RRlHom_{\mathscr X}(-,-)}\arrow[bend right = 80,swap]{dd}{\mathsf p^\ast\times \mathsf p^\ast}
& & \derived(\QCoh_{\mathscr X})\isoarrow{d} \arrow[bend left = 70]{dd}{\mathsf p^\ast} \\
\Perf^G(X)\times \derived(\QCoh_X^G)^{\op} \arrow{d}{\Phi\times\Phi} \arrow[dotted]{rr}{\RRlHom_{X}(-,-)}
& & \derived(\QCoh_X^G) \arrow{d}{\Phi}\\
\Perf(X)\times \derived(\QCoh_X)^{\op} \arrow{rr}{\RRlHom_{X}(-,-)}
& & \derived(\QCoh_X) 
\end{tikzcd}
\]
where the bottom row is the ordinary derived sheaf Hom functor. 
The top square is used to \emph{define} the $G$-equivariant $\RRlHom_{X}(-,-)$ in the middle row, so it commutes by construction, whereas the commutativity of the whole diagram, which is equivalent to the statement
\begin{equation}\label{RHom_commute_with_PB}
\mathsf p^\ast \RRlHom_{\mathscr X}(\mathscr E^\bullet,\mathscr F^\bullet) = \RRlHom_X(\mathsf p^\ast \mathscr E^\bullet,\mathsf p^\ast \mathscr F^\bullet),
\end{equation}
is used to observe that the lower square is also commutative: this means that the $G$-equivariant $\RRlHom_{X}(-,-)$ commutes with the forgetful functor $\Phi$.

We thank David Rydh for pointing out the following.

\begin{remark}
 The ``quasi-coherent'' derived sheaf Hom functor
\[
\RRlHom_{\mathscr X}^{\qcoh}(-,-)\colon \derived_{\qcoh}(\mathscr X)\times \derived_{\qcoh}(\mathscr X)^{\op} \to \derived_{\qcoh}(\mathscr X)
\]
considered in \cite[\S\,1.2]{Hall_Rydh_1}, defined by composing the restriction of \eqref{derived_rhom} with the \emph{quasi-coherator} $\derived(\mathscr X) \to \derived_{\qcoh}(\mathscr X)$, does not commute with pullback in general. However, \eqref{RHom_commute_with_PB} holds true precisely because we restricted to perfect complexes in the first entry. In this case, if we take $(\mathscr E^\bullet,\mathscr F^\bullet) \in \Perf(\mathscr X) \times \derived_{\qcoh}(\mathscr X)$,  we have indeed
\[
\RRlHom_{\mathscr X}(\mathscr E^\bullet,\mathscr F^\bullet)=\RRlHom^{\qcoh}_{\mathscr X}(\mathscr E^\bullet,\mathscr F^\bullet) = \mathscr E^{\bullet \vee}\otimes^{\LL}\mathscr F^\bullet,
\]
and \eqref{RHom_commute_with_PB} follows since duals of perfect complexes, which are dualisable, commute with arbitrary pullback. See \cite[Lemma 4.3]{Hall_Rydh_1} and the paragraph before it for more details.
\end{remark}

The $G$-equivariant derived functors listed above satisfy the usual compatibilities. Here are some of them: given a morphism $f\colon X\to Y$ of $G$-schemes,
\begin{enumerate}
\item $\LL f^\ast$ is left adjoint to $\RR f_\ast$, 
\item $\RR f_\ast$ preserves 
cohomologically bounded complexes,
\item $\LL f^\ast$, $\RR f_\ast$, ${\otimes}^{\LL}$ and $\RRlHom_X(-,-)$ commute with the forgetful functor $\Phi$ (cf.~\cite[Section~1.5.8]{VV}). \label{commute_with_Phi} 
\item The projection formula 
\[
\mathscr E^\bullet {\otimes}^{\LL} \,\RR f_\ast\mathscr F^\bullet = \RR f_\ast\bigl(\LL f^\ast \mathscr E^\bullet{\otimes}^{\LL} \mathscr F^\bullet \bigr)
\]
holds, for all $\mathscr F^\bullet \in \derived(\QCoh_X^G)$ and $\mathscr E^\bullet \in \derived(\QCoh_Y^G)$.
\end{enumerate} 

For us, the most important property is \eqref{commute_with_Phi}.

\subsection{Equivariant Ext groups}
For a $G$-scheme $X$ with structure morphism $\pi\colon X\to \Spec \BC$, we write $\RR\Gamma_X = \RR \pi_\ast$.
Since the $G$-equivariant derived functors commute with the forgetful morphism,
given $(\mathscr E^\bullet,\mathscr F^\bullet)\in \Perf^G(X) \times \derived(\QCoh_X^G)$, the complex
\[
\RR\Gamma_X \RRlHom_X(\mathscr E^\bullet,\mathscr F^\bullet) \,\in\,\derived(\QCoh_{\pt}^G) = \derived(\Rep_{\BC}(G))
\]
is a complex of $G$-representations with $\RHom_X(\Phi_X(\mathscr E^\bullet),\Phi_X(\mathscr F^\bullet))$ as underlying complex of vector spaces. We will often omit $\Phi_X$ from the notation.

\begin{remark}\label{rmk:G-structure_on_cohomology} 
The cohomology functors $\HH^k\colon \derived(\QCoh_{\pt}^G) \to \QCoh_{\pt}^G$ also commute with the forgetful functor. 
In other words, for any object 
$V^\bullet \in \derived(\QCoh_{\pt}^G)$, there is a natural structure of $G$-representation on the vector spaces $\HH^k(\Phi_{\pt}(V^\bullet))$.
Thus, given $(\mathscr E^\bullet,\mathscr F^\bullet)\in \Perf^G(X) \times \derived(\QCoh_X^G)$, all Ext groups
\begin{equation}\label{ext}
\Ext^k_X(\mathscr E^\bullet,\mathscr F^\bullet) \defeq \HH^k(\RHom_X(\Phi_X(\mathscr E^\bullet),\Phi_X(\mathscr F^\bullet)))\,\in\,\QCoh_{\pt}
\end{equation}
have a natural structure of $G$-representations. Therefore the $G$-invariant part
\[
\Ext^k_X(\mathscr E^\bullet,\mathscr F^\bullet)^G \subset \Ext^k_X(\mathscr E^\bullet,\mathscr F^\bullet)
\]
is well-defined.
\end{remark}

\smallbreak
We now describe the $G$-representation structure on \eqref{ext} explicitly. We set $k=0$, the general case being obtained by replacing $\mathscr F^\bullet$ with $\mathscr F^\bullet[k]$.

Fix a pair $(\mathscr E^\bullet,\mathscr F^\bullet)\in \Perf^G(X) \times \derived(\QCoh_X^G)$ and a morphism 
$\alpha\colon \mathscr E^\bullet \to \mathscr F^\bullet$ in $\derived(\QCoh_X)$. For simplicity, assume
$\alpha$ is represented by a cochain map    
\begin{equation}\label{chain_map}
\begin{tikzcd}
\cdots \arrow{r} & \mathscr E^i \arrow{r}{\dd_{\mathscr E^\bullet}^i} \arrow[swap]{d}{\alpha_i} & \mathscr E^{i+1}\arrow{r}\arrow{d}{\alpha_{i+1}} & \cdots \\
\cdots \arrow{r} & \mathscr F^i \arrow{r}{\dd_{\mathscr F^\bullet}^i} & \mathscr F^{i+1}\arrow{r} & \cdots
\end{tikzcd}
\end{equation}
where all arrows are in the category $\QCoh_X$. The sheaves $\mathscr E^i$ (resp.~$\mathscr F^i)$ carry a $G$-equivariant 
structure $\vartheta_{\mathscr E^i}$ (resp.~$\vartheta_{\mathscr F^i}$). We let $g \in G$ act on $\alpha = (\alpha_i)_{i \in \BZ}$ 
by $g\cdot \alpha = (g\cdot \alpha_i)_{i \in \BZ}$, where the element $g\cdot \alpha_i = \vartheta_{\mathscr F^{i},g}^{-1} \circ g^\ast \alpha_i \circ \vartheta_{\mathscr E^{i},g} \in \Hom_X(\mathscr E^i,\mathscr F^i)$ was defined via Diagram \eqref{diag:action_on_Homs}. For fixed $g \in G$ and $i \in \BZ$, the diagram
\begin{equation}\label{diag:G-action_chain_complex}
\begin{tikzcd}[row sep=large,column sep=large]
\mathscr E^i \arrow{r}{\dd_{\mathscr E^\bullet}^i}\arrow{d}{\vartheta_{\mathscr E^i,g}}\arrow[bend left=-425,swap]{ddd}[description]{g\cdot \alpha_i} 
& \mathscr E^{i+1}\arrow{d}{\vartheta_{\mathscr E^{i+1},g}}\arrow[bend right=-425]{ddd}[description]{g\cdot \alpha_{i+1}} \\
g^\ast \mathscr E^i \arrow{r}{g^\ast\dd_{\mathscr E^\bullet}^i}\arrow{d}{g^\ast \alpha_i}
&  g^\ast \mathscr E^{i+1}\arrow{d}{g^\ast\alpha_{i+1}} \\
g^\ast \mathscr F^i \arrow{r}{g^\ast\dd_{\mathscr F^\bullet}^i}\arrow{d}{\vartheta_{\mathscr F^i,g}^{-1}}
& g^\ast \mathscr F^{i+1} \arrow{d}{\vartheta_{\mathscr F^{i+1},g}^{-1}} \\
\mathscr F^i \arrow{r}{\dd_{\mathscr F^\bullet}^i} 
& \mathscr F^{i+1}
\end{tikzcd}
\end{equation}
in $\QCoh_X$ illustrates the situation: since $(\mathscr E^\bullet,\dd_{\mathscr E^\bullet})$ and $(\mathscr F^\bullet,\dd_{\mathscr F^\bullet})$ are objects of $\derived(\QCoh_X^G)$, the morphisms $\dd_{\mathscr E^\bullet}^i$ and $\dd_{\mathscr F^\bullet}^i$, as soon as we view them in $\QCoh_X$, satisfy 
\begin{align*}
\dd_{\mathscr E^\bullet}^i &= g\cdot \dd_{\mathscr E^\bullet}^i = \vartheta^{-1}_{\mathscr E^{i+1},g}\circ g^\ast \dd^i_{\mathscr E^\bullet}\circ \vartheta_{\mathscr E^{i},g}\\
\dd_{\mathscr F^\bullet}^i &= g\cdot \dd_{\mathscr F^\bullet}^i = \vartheta^{-1}_{\mathscr F^{i+1},g}\circ g^\ast \dd^i_{\mathscr F^\bullet}\circ \vartheta_{\mathscr F^{i},g}
\end{align*} 
respectively, for all $g \in G$ (cf.~Remark \ref{remark:G-inv_Homs}).
Therefore the top and bottom squares commute. So does the middle square, by the commutativity of \eqref{chain_map}. Therefore the outer square commutes, thus defining the morphism $g\cdot \alpha \in \Hom_X(\mathscr E^\bullet,\mathscr F^\bullet)$.

\begin{remark}\label{identity_is_equivariant}
Let $\mathscr E^\bullet$ be an object of $\Perf^G(X)$. Then
\[
\id_{\mathscr E^\bullet} \,\in\, \Hom_X(\mathscr E^\bullet,\mathscr E^\bullet)^G \subset \Hom_X(\mathscr E^\bullet,\mathscr E^\bullet).
\]
This is clear by looking at the diagram \eqref{diag:G-action_chain_complex} where all $\alpha_i = \id_{\mathscr E^i}$.
\end{remark}

\begin{lemma}\label{lemma:equiv_EXT}
The following statements hold.
\begin{enumerate}
\item
Fix $\mathscr F^\bullet \in \Perf^G(X)$. A distinguished triangle $\mathscr E_1^\bullet \to \mathscr E_2^\bullet \to \mathscr E_3^\bullet \to \mathscr E_1^\bullet[1]$ in $\derived(\QCoh_X^G)$ induces a long exact sequence
\[
\cdots \to \Ext^k_X(\mathscr F^\bullet,\mathscr E_1^\bullet) \to \Ext^k_X(\mathscr F^\bullet,\mathscr E_2^\bullet) \to \Ext^k_X(\mathscr E_3^\bullet,\mathscr F^\bullet)  \to \Ext^{k+1}_X(\mathscr F^\bullet,\mathscr E_1^\bullet) \to \cdots
\]
of $G$-representations. \label{statement1}
\item
Fix $\mathscr F^\bullet \in \derived(\QCoh_X^G)$. A distinguished triangle $\mathscr E_1^\bullet \to \mathscr E_2^\bullet \to \mathscr E_3^\bullet \to \mathscr E_1^\bullet[1]$ in $\Perf^G(X)$ induces a long exact sequence
\[
\cdots \to \Ext^k_X(\mathscr E_3^\bullet,\mathscr F^\bullet) \to \Ext^k_X(\mathscr E_2^\bullet,\mathscr F^\bullet) \to \Ext^k_X(\mathscr E_1^\bullet,\mathscr F^\bullet)  \to \Ext^{k+1}_X(\mathscr E_3^\bullet,\mathscr F^\bullet) \to \cdots
\]
of $G$-representations.\label{statement2}
\end{enumerate}
\end{lemma}

\begin{proof}
To prove \eqref{statement1}, apply the composition $\RR\Gamma_X\circ \RRlHom_X(\mathscr F^\bullet,-)$ of equivariant derived functors, and then cohomology 
$\HH^\bullet \colon \derived(\Rep_{\BC}(G)) \to \Rep_{\BC}(G)$, to the given distinguished triangle: exploiting Remark \ref{rmk:G-structure_on_cohomology}, this yields the first sequence. To prove \eqref{statement2}, use $\RRlHom_X(-,\mathscr F^\bullet)$.
\end{proof}

We will only need the following special case. 

\begin{corollary}\label{lemma:equivariant_composition}
The following statements hold.
\begin{enumerate}
\item
Fix $\mathscr F^\bullet \in\Perf^G(X)$. A morphism $j\colon \mathscr E_1^\bullet \to \mathscr E_2^\bullet$ in $\derived(\QCoh_{X}^G)$ induces a morphism of $G$-representations
\[
j_\ast  \colon \Hom_X( \mathscr F^\bullet,\mathscr E_1^\bullet ) \to \Hom_X( \mathscr F^\bullet,\mathscr E_2^\bullet ),\quad \alpha \mapsto  j\circ \alpha.
\]
\item
Fix $\mathscr F \in \derived(\QCoh_X^G)$. A morphism $i\colon \mathscr E_1^\bullet \to \mathscr E_2^\bullet$ in $\Perf^G(X)$ induces a morphism of $G$-representations
\[
i^\ast \colon \Hom_X( \mathscr E_2^\bullet,\mathscr F^\bullet ) \to \Hom_X( \mathscr E_1^\bullet,\mathscr F^\bullet ), \quad \beta \mapsto \beta\circ i.
\]
\label{item84762874}
\end{enumerate}
\end{corollary}
 
 The following definition will be central in the next sections.

\begin{definition}\label{G-equiv-hom}
Fix $(\mathscr E^\bullet,\mathscr F^\bullet) \in\derived(\QCoh_X^G) \times \derived(\QCoh_X^G)$. We say that an extension class
\[
\alpha \,\in\,\Ext^k_X(\mathscr E^\bullet,\mathscr F^\bullet)
\]
is $G$-\emph{equivariant} if the corresponding morphism $\mathscr E^\bullet \to \mathscr F^\bullet[k]$ admits a lift to $\derived(\QCoh_X^G)$, i.e.~if it lies in the image of the natural morphism $\Hom_{\derived(\QCoh_X^G)}(\mathscr E^\bullet,\mathscr F^\bullet[k]) \to \Hom_{X}(\mathscr E^\bullet,\mathscr F^\bullet[k])$ --- recall that we omit $\Phi$ from the notation in the target Hom-set. We say that $\alpha$ is $G$-\emph{invariant} if it belongs to $\Ext^k_X(\mathscr E^\bullet,\mathscr F^\bullet)^G$.
\end{definition}

\subsubsection{The case of reductive groups}\label{subsec:reductives}
Recall from \cite[App.~A]{MFK} that, over a field of characteristic $0$, 
a linear algebraic group $G$ is \emph{reductive} if and only if it is linearly reductive. This means that the functor
\[
(-)^G\colon \QCoh_{\pt}^G \to \QCoh_{\pt}, \quad V\mapsto V^G,
\]
taking a $G$-representation to its $G$-invariant part, is exact.

Reductivity has the following important property.

\begin{lemma}[{\cite[Lemma 2.2.8]{BFK}}]\label{lemma:G-inv_Ext_groups}
Let $G$ be a reductive algebraic group acting on a complex noetherian separated $G$-scheme $X$. Set  $\mathscr X = [X/G]$ and fix $k \in \BZ$ and two objects $\mathscr E^\bullet$, $\mathscr F^\bullet \in \derived(\QCoh_X^G) = \derived(\QCoh_{\mathscr X})$. Then there are natural isomorphisms
\[
\Hom_{\mathscr X}(\mathscr E^\bullet,\mathscr F^\bullet[k]) \,\simto\, \Hom_{X}(\mathscr E^\bullet,\mathscr F^\bullet[k])^{G}.
\]
\end{lemma}

The Hom-set on the left hand side is taken in $\derived(\QCoh_{\mathscr X})$, the Hom-set on the right hand side is taken in $\derived(\QCoh_X)$.

\begin{remark}
If $G$ is reductive, then by Lemma \ref{lemma:G-inv_Ext_groups} an extension class $\alpha$ is $G$-equivariant if and only if it is $G$-invariant.
\end{remark}

\subsection{Equivariant Grothendieck duality}\label{sec:verdier}

Classically, one says that Grothendieck duality holds for a morphism of schemes $f$ if the right derived functor $\RR f_\ast$ has a right adjoint. 
Such adjoint is usually denoted $f^{\times}$, or $f^!$ if $f$ is a proper morphism. We will stick to the $f^{!}$ notation.

The most general statement we are aware of is due to Neeman. Note that this is stated for $\derived_{\qcoh}$ in \cite{Neeman2}, but (as we observed in \S\,\ref{sec:perfect_complexes_on_schemes}) with our assumptions these categories are equivalent to $\derived(\QCoh)$.

\begin{theorem}[Grothendieck duality {\cite{Neeman2}}]\label{thm:classical_Grothendieck}
Let $f\colon X\to Y$ be a morphism of quasi-compact separated schemes. Then $\RR f_\ast \colon \derived(\QCoh_X) \to \derived(\QCoh_Y)$ has a right adjoint $f^!$.
If $f$ is a proper morphism of noetherian separated schemes, the natural morphism
\begin{equation}\label{iso_sheafy_VD}
\RR f_\ast \RRlHom_X(\mathscr F^\bullet,f^!\mathscr E^\bullet) \to \RRlHom_Y(\RR f_\ast \mathscr F^\bullet,\mathscr E^\bullet)
\end{equation}
is an isomorphism in $\derived(Y)$ for all $\mathscr F^\bullet \in \derived(\QCoh_X)$ and $\mathscr E^\bullet \in \derived(\QCoh_Y)$.
\end{theorem}

\begin{proof}
The first assertion is \cite[Ex.~4.2]{Neeman2}. The sheafified Grothendieck duality isomorphism \eqref{iso_sheafy_VD} is obtained in \cite[\S\,6]{Neeman2}.
A proof of \eqref{iso_sheafy_VD} assuming $f$ is a morphism essentially of finite type between noetherian separated  schemes can be found in \cite[Eq.~1.6.1]{MR3334894}.
\end{proof}

We refer the reader to Neeman \cite{Neeman2} and Lipman \cite{Lipman1} for very informative discussions around the 
history of Grothendieck duality, as well its more modern versions.

\smallbreak
In this section we prove a $G$-equivariant version of Theorem \ref{thm:classical_Grothendieck}. We follow Neeman's strategy entirely.
See also \cite[Thm.~4.14\,(1)]{Hall_Rydh_1} for a generalisation, proving the existence of a right adjoint of $\RR h_\ast\colon \derived_{\qcoh}(\mathscr X) \to \derived_{\qcoh}(\mathscr Y)$ for $h\colon \mathscr X \to \mathscr Y$ an arbitrary concentrated morphism (cf.~\cite[Def.~2.4]{Hall_Rydh_1}) of algebraic stacks.

The main tool used by Neeman is the following version of Brown's representability theorem.

\begin{theorem}[Brown representability~{\cite[Thm.~4.1]{Neeman2}}]\label{thm:Brown_Rep}
Let $\mathcal S$ be a compactly generated triangulated category, $\mathcal T$ any triangulated category. Let $F\colon \mathcal S \to \mathcal T$ be a triangulated functor respecting coproducts. Then $F$ has a right adjoint.
\end{theorem}

\begin{theorem}[Equivariant Grothendieck duality]\label{thm:equivariant_Grothendieck}
Let $f\colon X\to Y$ be a morphism of noetherian separated $G$-schemes satisfying $(\dagger)$.
Then $\RR f_\ast\colon \derived(\QCoh_X^G) \to \derived(\QCoh_Y^G)$ has a right adjoint $f^!$.
\end{theorem}

\begin{proof}
Recall that $\derived(\QCoh_X^G)$ is compactly generated by Corollary \ref{cor:derived_X^GcompactlyGenerated}.
Set $\mathscr X = [X/G]$ and $\mathscr Y = [Y/G]$. 
The morphism $f\colon X \to Y$ induces a \emph{representable} morphism of algebraic stacks
\[
\overline f\colon \mathscr X \to \mathscr Y,
\]
that by our assumptions on $X$ and $Y$ is quasi-compact and quasi-separated.
In particular, by \cite[Lemma~2.5]{Hall_Rydh_1}, $\overline f$ is a  \emph{concentrated} morphism.
Then, by \cite[Thm.~2.6\,(3)]{Hall_Rydh_1}, the direct image $\RR \overline{f}_{\ast}\colon \derived_{\qcoh}(\mathscr X)\to \derived_{\qcoh}(\mathscr Y)$ preserves coproducts. Under the equivalences \eqref{equivalences}, the functor $\RR \overline{f}_{\ast}$ corresponds precisely to $\RR f_\ast\colon \derived(\QCoh_X^G) \to \derived(\QCoh_Y^G)$.
Thus the existence of $f^!\colon \derived(\QCoh_Y^G) \to \derived(\QCoh_X^G)$ follows by Theorem \ref{thm:Brown_Rep}.
\end{proof}

\begin{lemma}[Sheafified Grothendieck duality]
Let $f\colon X\to Y$ be a proper morphism of noetherian separated $G$-schemes satisfying $(\dagger)$.
Fix objects $\mathscr F^\bullet \in \Perf^G(X)$ and $\mathscr E^\bullet \in \derived(\QCoh_Y^G)$. Then there is a natural isomorphism
\begin{equation}\label{sheafy_VD}
\RR f_{\ast}\RRlHom_{X}(\mathscr F^{\bullet},f^{!}\mathscr E^{\bullet}) \,\simto\, \RRlHom_{Y}(\RR f_{\ast}\mathscr F^{\bullet},\mathscr E^{\bullet})
\end{equation}
in $\derived(\QCoh_Y^G)$.
\end{lemma}

\begin{proof}
This is a special case of \cite[Lemma 5.3]{Neeman0}.
\end{proof}

One can ask whether the right adjoint $f^!\colon \derived(\QCoh_Y^G) \to \derived(\QCoh_X^G)$ commutes with the forgetful functor.
This question can be restated as follows. Given the $2$-cartesian diagram
\[
\begin{tikzcd}[column sep=large,row sep=large]
X\MySymb{dr}\arrow{r}{f}\arrow[swap]{d}{\mathsf p_X} & Y\arrow{d}{\mathsf p_Y} \\
\mathscr X \arrow{r}{\overline f}  & \mathscr Y
\end{tikzcd}
\]
we ask whether the natural transformation $\eta\colon \mathsf p_X^\ast \overline f^! \to f^! \mathsf p_Y^\ast$ an isomorphism of functors. 
This is answered in full generality in \cite[Lemma 5.20]{Neeman0}.
For the purpose of this paper, we content ourselves with a special case of that result: the answer is positive, i.e.~$\eta_{\mathscr E^\bullet}$ is an isomorphism, when $f$ is proper 
and $\mathscr E^\bullet \in \derived_{\qcoh}^+(\mathscr Y)$ is bounded below. Under these assumptions one has
\begin{equation}\label{eqn:for}
\Phi_X(f^!(\mathscr E^\bullet)) = f^!(\Phi_Y(\mathscr E^\bullet))
\end{equation}
in $\derived(\QCoh_X)$, where $f^!$ in the left hand side (resp.~in the right hand side) is the $G$-equivariant right adjoint (resp.~the ordinary right adjoint) of $\RR f_\ast$.

\begin{remark}
In a little more detail, properness of $f$ implies properness of $\overline f$ (reason: $\overline f$ is separated by \cite[\href{https://stacks.math.columbia.edu/tag/04YV}{Tag 04YV}]{stacks-project}, universally closed by \cite[\href{https://stacks.math.columbia.edu/tag/0CL3}{Tag 0CL3}]{stacks-project} and of finite type by \cite[\href{https://stacks.math.columbia.edu/tag/06FR}{Tag 06FR}]{stacks-project}); since $\mathscr X$ and $\mathscr Y$ are noetherian and $\overline f$ is representable, it follows that $\overline f$ is \emph{quasi-proper}, which together with $\mathscr E^\bullet \in \derived_{\qcoh}^+(\mathscr Y)$ is one of the conditions in \cite[Lemma 5.20]{Neeman0} ensuring that  $\eta_{\mathscr E^\bullet}$ is an isomorphism.

Note that restricting attention to bounded below complexes does not affect the applications we have in mind, which involve perfect complexes: we already observed in \S\,\ref{sec:perfect_complexes_on_schemes} that by assumption $(\dagger)$ perfect complexes are \emph{bounded}.
\end{remark}

\begin{corollary}\label{cor:Equivariant_Grothendieck_Global_sections}
Let $f\colon X\to Y$ be a proper morphism of noetherian separated $G$-schemes satisfying $(\dagger)$. 
Given $\mathscr F^\bullet \in \Perf^G(X)$ and $\mathscr E^\bullet \in \derived_{\qcoh}^+(\mathscr Y)\subset  \derived(\QCoh_Y^G)$, 
for all $k \in \BZ$ there is a canonical isomorphism of $G$-representations
\[
\Ext^k_X(\mathscr F^{\bullet},f^{!}\mathscr E^{\bullet}) \,\simto\, \Ext^k_Y(\RR f_{\ast}\mathscr F^{\bullet},\mathscr E^{\bullet}).
\]
Taking $G$-invariant parts, it restricts to an isomorphism of $\BC$-vector spaces
\[
\Ext^k_X(\mathscr F^{\bullet},f^{!}\mathscr E^{\bullet})^G \,\simto\, \Ext^k_Y(\RR f_{\ast}\mathscr F^{\bullet},\mathscr E^{\bullet})^G.
\]
\end{corollary}

\begin{proof}
It is enough to apply $\HH^k\circ \RR \Gamma_Y$ to the isomorphism \eqref{sheafy_VD} and to observe that all functors involved commute with the forgetful functor. For $f^!$, we exploit  \eqref{eqn:for}. 
\end{proof}

\begin{example}
Keep the assumptions of Corollary \ref{cor:Equivariant_Grothendieck_Global_sections}. If $G$ is reductive, by Lemma \ref{lemma:G-inv_Ext_groups} we have a commutative diagram of isomorphisms
\[
\begin{tikzcd}
\Ext^k_X(\mathscr F^{\bullet},f^{!}\mathscr E^{\bullet})^G\arrow{r}{\sim} & \Ext^k_Y(\RR f_{\ast}\mathscr F^{\bullet},\mathscr E^{\bullet})^G \\
\Hom_{\derived(\QCoh_X^G)}(\mathscr F^{\bullet},f^{!}\mathscr E^{\bullet}[k])\isoarrow{u} \arrow{r}{\sim} & \Hom_{\derived(\QCoh_Y^G)}(\RR f_{\ast}\mathscr F^{\bullet},\mathscr E^{\bullet}[k])\isoarrow{u}
\end{tikzcd}
\]
where the bottom map is the adjunction isomorphism obtained via Theorem \ref{thm:equivariant_Grothendieck}.
\end{example}

\section{Equivariance of the truncated Atiyah class}\label{sec:Equivariance_of_At}

\subsection{Truncated Atiyah classes after Huybrechts--Thomas}
In this section all schemes are noetherian and separated over $\BC$.

\subsubsection{The relative truncated cotangent complex}\label{sec:truncated_cotg}
The goal of this subsection is to revisit the classical fact that the truncated cotangent complex, though defined through the choice of a smooth embedding, does not depend on this choice. We review this from \cite[\S\,2]{HT} since the argument reveals that the same feature occurs in the equivariant setting.

Let $B$ be a scheme. Let $X \subset A_1$ be a closed embedding inside a smooth $B$-scheme $A_1$. Let $J_1 \subset \OO_{A_1}$ be the ideal sheaf of the embedding. Consider the exterior derivative
\[
\dd \colon J_1 \into \OO_{A_1} \to \Omega_{A_1/B}
\]
and restrict it to $X$ to obtain the (relative) \emph{truncated cotangent complex}
\begin{equation}\label{ordinary_L}
\BL_{X/B} \,=\,\bigl[\,J_1/J_1^2 \to \Omega_{A_1/B}|_X\,\bigr] \,\in \,\derived^{[-1,0]}(\QCoh_X).
\end{equation}
Suppose $X$ admits an embedding in another smooth $B$-scheme $A_2$. Then the composition $X \into A_1\times_B A_2 \to A_1$, where $X \into A_1 \times_BA_2$ is the diagonal embedding defined by an ideal $J_{12} \subset \OO_{A_1\times_BA_2}$, induces a quasi-isomorphism of two-term complexes
\begin{equation}
\begin{tikzcd}\label{big_qis}
J_1/J_1^2 \arrow{r}\arrow{d} & \Omega_{A_1/B}|_X \arrow{d} \\
J_{12}/J_{12}^2\arrow{r}\arrow{d} & \Omega_{A_1/B}|_X\oplus \Omega_{A_2/B}|_X \arrow{d} \\
\Omega_{A_2/B}|_X\arrow[equal]{r} &  \Omega_{A_2/B}|_X
\end{tikzcd}
\end{equation}
showing that replacing $X \subset A_1$ with $X \subset A_1\times_BA_2$ does not change the isomorphism class of $\BL_{X/B}$ in the derived category.

\smallbreak
We recalled this argument in order to make the following observation.
Suppose $\iota_i\colon X\into A_i$ is a $G$-\emph{equivariant} closed embedding, for $i = 1,2$. Then 
\[
0 \to J_i \to \OO_{A_i} \to \iota_{i\ast}\OO_X \to 0
\]
is a $G$-equivariant short exact sequence, and similarly for $\iota_{12}\colon X \into A_1\times_BA_2$.
Since the exterior derivative $\dd \colon \OO_{A_i} \to \Omega_{A_i/B}$ is also $G$-equivariant, the whole diagram \eqref{big_qis} can be canonically lifted to $\QCoh_X^G$. This yields a well-defined element
\begin{equation}\label{equivariant_L}
\BL_{X/B}\,\in\,\derived^{[-1,0]}(\QCoh_X^G),
\end{equation}
whose isomorphism class again does not depend on the choice of equivariant embedding. The equivariant truncated cotangent complex is also discussed by Illusie in \cite[Ch.~VII, \S\,2.2.5]{IllusieII}.

Let $L^\bullet_{S/T}$ denote the full (possibly $G$-equivariant) cotangent complex of a (possibly $G$-equivariant) morphism of schemes (or algebraic stacks) $S\to T$.

\begin{lemma}\label{lemma:equivariant_transitivity_triangle}
Let $f\colon X \to Y$ and $g\colon Y \to Z$ be $G$-equivariant morphisms. Then there is a sequence of morphisms
\[
\tau_{\geq-1}\LL f^\ast L_{Y/Z}^\bullet \to \BL_{X/Z} \to \BL_{X/Y} 
\]
in $\derived^{[-1,0]}(\QCoh_X^G)$.
\end{lemma}

\begin{proof}
Let us shorten $\mathscr X = [X/G]$, $\mathscr Y=[Y/G]$ and $\mathscr Z=[Z/G]$. 
The given $G$-equivariant morphisms induce $2$-cartesian diagrams of algebraic stacks
\[
\begin{tikzcd}[row sep = large,column sep = large]
X\MySymb{dr}\arrow{d}{\mathsf p_X}\arrow{r}{f} & Y\MySymb{dr}\arrow{d}{\mathsf p_Y}\arrow{r}{g} & Z\arrow{d}{\mathsf p_Z} \\
\mathscr X \arrow{r}{\overline f} & \mathscr Y \arrow{r}{\overline g} & \mathscr Z
\end{tikzcd}
\]
where the morphisms $\overline f$ and $\overline g$ are quasi-compact, quasi-separated and of Deligne--Mumford type \cite[\href{https://stacks.math.columbia.edu/tag/04YW}{Tag 04YW}]{stacks-project}. Hence their cotangent complexes live in $\derived^{\leq 0}(\QCoh)$. By \cite[Thm.~8.1]{Olsson}, there is an exact triangle
\begin{equation}\label{triangle}
\LL \overline f^\ast L^\bullet_{\mathscr Y/\mathscr Z} \to L^\bullet_{\mathscr X/\mathscr Z} \to L^\bullet_{\mathscr X/\mathscr Y} 
\end{equation}
in $\derived(\QCoh_{\mathscr X})$, where $L^\bullet$ denotes the full cotangent complex. If we applied the pullback functor
\[
\mathsf p_X^\ast\colon \derived(\QCoh_{\mathscr X}) \,\simto\, \derived(\QCoh_X^G) \xrightarrow{\Phi} \derived(\QCoh_X)
\] 
to the triangle \eqref{triangle}, we would get the usual triangle of full cotangent complexes
\begin{equation}\label{triangle2}
\LL f^\ast L_{Y/Z}^\bullet \to L_{X/Z}^\bullet \to L_{X/Y}^\bullet 
\end{equation}
in $\derived^{\leq 0}(\QCoh_X)$. Instead, we get a lift to $\derived(\QCoh_X^G)$ of the triangle \eqref{triangle2}
by applying the exact equivalence $\derived(\QCoh_{\mathscr X}) \,\simto\, \derived(\QCoh_X^G)$ to \eqref{triangle}. Applying the truncation functor $\tau_{\geq -1}$ on $\derived(\QCoh_X^G)$ yields the desired sequence of morphisms
\[
\tau_{\geq-1}\LL f^\ast L_{Y/Z}^\bullet \to \BL_{X/Z} \to \BL_{X/Y}
\]
in $\derived^{[-1,0]}(\QCoh_X^G)$, as required.
\end{proof}

\subsubsection{Absolute setting}\label{sec:absolute} 
Let $X \into A$ be a closed immersion of a scheme $X$ inside a smooth $\BC$-scheme $A$. Let $J \subset \OO_A$ be the corresponding sheaf of ideals. The (absolute) truncated cotangent complex is the two term complex
\begin{equation}\label{truncated_cotg_complex}
\BL_X = \bigl[ J/J^2 \to \Omega_A|_X\bigr] \,\in\,\derived^{[-1,0]}(\QCoh_X).
\end{equation}
Let $\mathscr I_A\subset \OO_{A\times A}$ and $\mathscr I_X \subset \OO_{X\times X}$ be the ideal sheaves of the diagonal embeddings
\[
A \xhookrightarrow{i_{\Delta_A}} A\times A,\quad X \xhookrightarrow{i_{\Delta_X}} X\times X,
\]
respectively. Huybrechts--Thomas \cite[\S\,2]{HT} show how to construct a canonical morphism
\begin{equation}\label{basic_morphism_FM}
\alpha_X\colon \OO_{\Delta_X} \to i_{\Delta_X\ast}\BL_{X}[1].
\end{equation}
It is represented in degrees $[-2,0]$ by the morphism of complexes
\begin{equation}\label{complex_defining_alpha}
\begin{tikzcd}
i_{\Delta_X\ast}(J/J^2) \arrow[equal]{d}\arrow{r} &
\mathscr I_A\big|_{X\times X}\arrow{d}\arrow{r} &
\OO_{X\times X} \\
i_{\Delta_X\ast}(J/J^2)\arrow{r} & 
\mathscr I_A/\mathscr I_A^2 \big|_{X\times X}
\end{tikzcd}
\end{equation}
where the quasi-isomorphism between the top complex and $\OO_{\Delta_X}$ is proved as a consequence of \cite[Lemma 2.2]{HT}.
The extension class 
\[
\alpha_X \in \Ext_{X\times X}^1(\OO_{\Delta_X},i_{\Delta_X\ast}\BL_{X})
\]
corresponding to \eqref{basic_morphism_FM} is called the \emph{truncated universal Atiyah class}. It does not depend on the choice of embedding $X \subset A$.

The main observation in \cite{HT}, at this point, is that the map \eqref{basic_morphism_FM} 
can be seen as a map of Fourier--Mukai kernels. In particular, for a perfect complex $E$ on $X$, one can view $\RR\pi_{2\ast}(\pi_1^{\ast}E\otimes \alpha_X)$ 
as a canonical morphism
\[
\At_E\colon E \to E \otimes \BL_{X}[1]
\]
in $\derived(\QCoh_X)$, where $\pi_i\colon X\times X\to X$ are the projections.
This is, by definition, the \emph{truncated Atiyah class} of $E$ introduced in \cite[Def.~2.6]{HT}. It can of course be seen as an element
\begin{equation}\label{def:truncated_atiyah_class}
\At_E \in \Ext_X^1(E,E\otimes \BL_{X}).
\end{equation}
Under the canonical morphism $\BL_{X} \to h^0(\BL_{X}) = \Omega_{X}$,
the extension $\At_E$ projects onto the classical Atiyah class in $\Ext_X^1(E,E\otimes \Omega_{X})$.

\subsubsection{Relative setting}\label{sec:relative} 
We consider the following setup, which we recall verbatim from \cite[\S\,2]{HT} and \cite{HT_erratum}. 

Let $B$ be a scheme, $X \to B$ a $B$-scheme equipped with a closed immersion $X \into A$ with ideal $J \subset \OO_A$.
We assume we have a commutative diagram
\begin{equation}
\begin{tikzcd}\label{smooth_embeddings}
X \arrow[hook]{r}\arrow{dr} 
& A_B \MySymb{dr} \arrow[hook]{r} \arrow{d}
& A\arrow{d} \\
& 
B\arrow[hook]{r} 
& \widetilde{B}
\end{tikzcd}
\end{equation}
where $\widetilde B$ and $A \to \widetilde B$ are smooth and the square is cartesian. In particular, both $A$ and $A_B \to B$ are smooth.
Let $J_B \subset \OO_{A_B}$ be the ideal sheaf of $X \subset A_B$. Then there is a natural morphism of chain complexes
\begin{equation}\label{dia99}
\begin{tikzcd}
J/J^2 \arrow{r}\arrow{d}
& J_B/J_B^2 \arrow{d} & \\
\Omega_A\big|_X \arrow{r}
& \Omega_{A/\widetilde B}\big|_X\arrow[equal]{r}
& \Omega_{A_B/B}\big|_X
\end{tikzcd}
\end{equation}
inducing a morphism
\[
j\colon \BL_X \to \BL_{X/B}.
\]
The \emph{relative truncated Atiyah class} of a perfect complex $E \in \Perf X$ is, by definition, the composition
\[
\At_{E/B} \colon E \xrightarrow{\At_E} E\otimes \BL_X[1] \xrightarrow{\id_E \otimes j[1]} E\otimes \BL_{X/B}[1].
\]
It corresponds to the element 
\[
\At_{E/B}\,\in\,\Ext_X^1(E,E\otimes \BL_{X/B})
\]
obtained as the image of $\At_E$ under the map $(\id_E \otimes j[1])_\ast\colon \Ext_X^1(E,E\otimes \BL_{X}) \to \Ext_X^1(E,E\otimes \BL_{X/B})$.

\subsection{Adding in the group action}

In this section we prove Theorem \ref{main_thm} (which builds on the situation of \S\,\ref{sec:absolute}), along with its relative analogue (which builds on the situation of \S\,\ref{sec:relative}).

\subsubsection{Absolute setting}\label{absolute_again}
We first go back to the absolute setting of \S\,\ref{sec:absolute}. 

Let $G$ be an affine algebraic group, and let 
\[
X \subset A
\]
be a $G$-\emph{equivariant embedding} of noetherian separated schemes, where $A$ is smooth. Recall (cf.~Example \ref{rmk:G-equiv_emb}) that this situation is achieved if $X$ is quasi-projective and has a $G$-equivariant line bundle. 
Under these assumptions, we have seen that the truncated cotangent complex is canonically $G$-equivariant, i.e.~there is a canonical lift
\[
\BL_{X} \,\,\in\,\,\derived^{[-1,0]}(\QCoh_X^G)
\]
of the complex \eqref{truncated_cotg_complex}.

Let $i_{\Delta_X}\colon X \into X\times X$ be the diagonal embedding. 
The $G$-action on $X$ determines a $G$-equivariant structure on the structure sheaf $\OO_X$ (Example \ref{ex:equivariant_structure_sheaf}), and on the short exact sequence
\[
\mathscr I_{X} \into \OO_{X\times X} \onto\OO_{\Delta_X} = i_{\Delta_X\ast}\OO_X.
\]

\begin{lemma}\label{lemma:alpha_X_equivariant}
The morphism
\[
\alpha_X\colon \OO_{\Delta_X} \to i_{\Delta_X\ast} \BL_{X}[1]
\]
is naturally $G$-equivariant.
\end{lemma}

\begin{proof}
Since $X\subset A$ is a $G$-equivariant embedding, the diagram of closed immersions
\[
\begin{tikzcd}[row sep=large]
X\arrow[hook,swap]{d}{i_{\Delta_X}}\arrow[hook]{r} & A\arrow[hook]{d}{i_{\Delta_A}} \\
X\times X \arrow[hook]{r} & A\times A
\end{tikzcd}
\]
along with its associated ideal sheaf short exact sequences, are also $G$-equivariant in a natural way. Therefore Diagram \eqref{complex_defining_alpha}, which is built out of these equivariant short exact sequences through the $G$-equivariant geometric functors (cf.~\S\,\ref{sec:underived_functors}), inherits a $G$-equivariant structure. But Diagram \eqref{complex_defining_alpha} represents precisely $\alpha_X$. The claim follows. 
\end{proof}

We finally have all the tools to complete the proof of Theorem \ref{main_thm}. 

\begin{theorem}\label{main_thm_BODY}
Let $G$ be a complex affine algebraic group acting on a separated noetherian $\BC$-scheme $X$ admitting a $G$-equivariant embedding in a smooth $G$-scheme. 
Fix a perfect complex $E \in \Perf X$. Then every lift of $E$ to $\derived(\QCoh_X^G)$ makes $\At_E$ canonically $G$-equivariant. 
\end{theorem}

Recall (cf.~Definition \ref{G-equiv-hom}) that by `$\At_E$ is $G$-equivariant' we mean that the corresponding morphism $E \to E \otimes \BL_X[1]$ admits a lift to $\derived(\QCoh_X^G)$.

\begin{proof}
Endow $X\times X$ with the diagonal action. Then the projections $\pi_i\colon X\times X \to X$ are $G$-equivariant and both $X$ and $X\times X$ satisfy Condition $(\dagger)$ by Lemma \ref{lemma:equivariant_embedding_implies_dagger}. Since $\alpha_X$ is $G$-equivariant by Lemma \ref{lemma:alpha_X_equivariant}, using equivariant pushforward $\RR \pi_{2\ast}$, pullback $\pi_1^\ast$ and tensor product $\otimes$ (cf.~\S\,\ref{sec:equiv_der_functors}), we deduce that the morphism
\[
\RR \pi_{2\ast}(\pi_1^\ast E {\otimes} \alpha_X)\colon E \to E \otimes \BL_{X}[1]
\]
is canonically lifted to $\derived(\QCoh_X^G)$, 
which proves the result.
\end{proof}

\subsubsection{Relative setting}\label{relative_again}

Suppose we are in the situation depicted in Diagram \eqref{smooth_embeddings}, and assume
$X \into A_B \into A$ are $G$-equivariant embeddings. Then we obtain the following consequence of Theorem \ref{main_thm_BODY}.

\begin{corollary}
The relative truncated Atiyah class $\At_{E/B}$ is $G$-equivariant.
\end{corollary}

\begin{proof}
The assumption that $X \into A_B \into A$ are $G$-equivariant implies that the morphism $j\colon \BL_X \to \BL_{X/B}$, induced by Diagram \eqref{dia99}, is $G$-equivariant. Therefore
\[
\At_{E/B}\colon E \to E\otimes \BL_X[1] \xrightarrow{\id_E\otimes j[1]} E\otimes \BL_{X/B}[1]
\]
lives in $\derived(\QCoh_X^G)$ entirely. 
\end{proof}

\begin{remark}
By Corollary \ref{lemma:equivariant_composition}, the morphism $\At_{E/B}$ can be seen as an element of 
\[
\Ext^1(E,E\otimes \BL_{X/B})^G \subset \Ext^1(E,E\otimes \BL_{X/B}).
\]
Indeed, since both $\At_E$ and $\id_E \otimes j[1]$ are morphisms in $\derived(\QCoh_X^G)$, the composition
\[
\Hom_X(E,E) \xrightarrow{\At_{E\ast}} \Hom_X(E,E\otimes \BL_{X}[1]) \xrightarrow{(\id_E \otimes j[1])_\ast} \Hom_X(E,E\otimes \BL_{X/B}[1])
\]
is a morphism of $G$-representations, and as such it preserves $G$-invariant parts. Therefore $\id_E \in \Hom_X(E,E)^G$ gets sent to $\At_{E/B} \in \Ext^1(E,E\otimes \BL_{X/B})^G$.
\end{remark}

\section{Application to moduli spaces of perfect complexes}\label{sec:perfect_complexes}

In this section we shall prove Theorem \ref{thm:main2}, whose statement we recall below (Theorem \ref{main2_body}) for the reader's convenience.

\subsection{The moduli space of perfect complexes}\label{sec:moduli_M}
Fix an affine algebraic group $G$ and a noetherian separated $\BC$-scheme $B$ carrying the trivial $G$-action $p_2\colon G\times B \to B$.
Let $f\colon Y \to B$ be a smooth (connected) projective $G$-invariant morphism of relative dimension $d$, 
where the $G$-action on $Y$ is denoted $\sigma_Y \colon G\times Y \to Y$. By assumption, $G$ preserves the fibres of $f$.

As in \cite[\S\,4.1]{HT}, let $M \to B$ be a relative fine separated moduli space of simple perfect complexes of rank $r \neq 0$ on the fibres of $f$, 
with fixed determinant $\mathcal L \in \Pic Y$ and fixed numerical invariants. Then $M$ is an algebraic space, locally complete as a moduli space, 
and there is a universal perfect complex
\[
E\,\in\,\Perf(Y\times_BM).
\]
Denote by $\iota_b\colon Y_b \into Y$ the inclusion of a fibre of $f$. If a point $m \in M$ sits over $b \in B$, let $i_m\colon Y_b \,\simto\,Y_b \times \set{m} \into Y\times_BM$ denote the corresponding inclusion.

For a scheme $S\to B$, the universal property of the pair $(M,E)$ translates into a bijection between
\begin{itemize}
\item morphisms $S \to M$ over $B$, and
\item equivalence classes of complexes $F \in \Perf (Y\times_BS)$ such that for all $s\in S$ (say, sitting over $b \in B$) the derived restriction $F|_{Y_b}$ is isomorphic to $\LL i_m^\ast E$ for some $m \in M$ (sitting over $b$), and such that $\det F = \pi_S^\ast \mathcal L'\otimes \pi_Y^\ast \mathcal L$ for some $\mathcal L' \in \Pic S$ (where $\pi_S$ and $\pi_Y$ are the projections from $Y\times_BS$).
\end{itemize}

Two complexes $F$ and $F'$ in $\Perf (Y\times_BS)$ are considered equivalent if there exists a line bundle $\mathcal H \in \Pic S$ such that $F = F' \otimes  \pi_S^\ast \mathcal H$. The correspondence assigns to a $B$-morphism $h\colon S \to M$ the equivalence class of the perfect complex $(\id_Y \times h)^\ast E \in \Perf(Y\times_BS)$.

\subsection{Equivariance of the universal complex}\label{sec:equiv_uni_complex}

From now on we set $X = Y\times_BM$. Note that $X$ does not have a $G$-action yet. In the next proposition, we construct such a $G$-action and we prove that the universal complex is equivariant. Before doing so, we state a fact that we will need during the proof.

\begin{fact}\label{rmk:Rosenlicht_trick}
A theorem of Rosenlicht \cite[Thm.~2]{Rosenlicht}, whose proof is sketched in \cite[Rem.~7.1]{Dolgachev_GIT}, says that if $Z$ and $Z'$ are irreducible varieties over an algebraically closed field, the natural homomorphism $\OO(Z)^\times \otimes \OO(Z')^\times \to \OO(Z\times Z')^\times$ is surjective. In fact, \cite[Rem.~7.1]{Dolgachev_GIT} shows more: one can write every function $\alpha \in \OO(Z\times Z')^\times$ as $\alpha = \beta\boxtimes \beta'$ for $\beta \in \OO(Z)^\times$ and $\beta' \in \OO(Z')^\times$. 
\end{fact}

\begin{prop}\label{prop:At_E_is_equiv}
The universal complex $E \in \Perf X$ is naturally $G$-equivariant.
\end{prop}

\begin{proof}
First of all, we lift the $G$-action $\sigma_Y\colon G\times Y \to Y$ to a $G$-action on $M$. Pulling back $E$ along
\[
\sigma_Y \times \id_M\colon G\times Y\times_BM \to Y\times_BM
\]
gives a family of perfect complexes parameterised by $G\times M$. By the universal property of $(M,E)$, this induces a $B$-morphism
\[
\sigma_M\colon G\times M \to M,
\]
which is a $G$-action on $M$. We have
\[
(\id_Y \times \sigma_M)^\ast E \,\cong\, (\sigma_Y \times \id_M)^\ast E\otimes \pi_{G,M}^\ast \mathcal H
\]
for some $\mathcal H \in \Pic(G\times M)$, where $\pi_{G,M}\colon G\times Y\times_BM \to G\times M$ is the projection. We claim that $\mathcal H$ is the trivial line bundle. Consider the projection $\pi_2\colon G\times M \to M$. Since $G$ is smooth and affine, we have $\Pic G = 0$, thus $\mathcal H = \pi_2^\ast \mathcal H'$ for some $\mathcal H' \in \Pic M$. However, $\mathcal H|_{\set{g}\times M}$ is trivial for all $g \in G$, in particular for $g = e$, where $e \in G$ is the group identity. Thus $\mathcal H'$ is trivial and hence so is $\mathcal H$. 
Then the previous isomorphism becomes
\[
(\id_Y \times \sigma_M)^\ast E \,\cong\, (\sigma_Y \times \id_M)^\ast E.
\]
Next, we have to make $E$ equivariant. We consider the $G$-action
\[
\tau\colon G\times Y\times_BM \to Y\times_BM,\quad (g,y,m)\mapsto (\sigma_Y(g,y),\sigma_M(g^{-1},m))
\]
on $X = Y\times_BM$. The pullback $\tau^\ast E$ corresponds to a $B$-morphism $\phi_\tau\colon G\times M \to M$. In fact, $\phi_\tau$ is the second projection. Indeed,
\[
\tau^\ast E|_{\set{g}\times Y_b \times\set{m}} = \LL i_n^\ast E,
\]
where $n = \sigma_M(g,\sigma_M(g^{-1},m)) = \sigma_M(e,m) = m$. Thus $\tau^\ast E|_{\set{g}\times Y_b \times\set{m}} = \LL i_m^\ast E$, and we obtain an isomorphism $\tau^\ast E \cong (\id_Y \times \phi_\tau)^\ast E \otimes \pi_{G,M}^\ast \mathcal H$ for some $\mathcal H \in \Pic(G\times M)$. For the same reason as before, $\mathcal H$ is trivial. Therefore, since $\phi_\tau$ is the projection, we obtain an isomorphism
\begin{equation}\label{equiv_str_E}
\vartheta \colon p_2^\ast E \,\simto\,\tau^\ast E
\end{equation}
of perfect complexes on $G\times X$, where $p_2 = \id_Y \times \phi_\tau\colon G\times X \to X$ is the projection. 

Finally, we need to verify that $\vartheta$ satisfies the cocycle condition. We follow \cite[Prop.~4.4]{Kool_Fixed_Point_Loci}, but we have to adapt the argument because the universal complex $E$ is not necessarily simple. By \cite[Prop.~2.4]{Kool_Fixed_Point_Loci}, it is enough to check the cocycle condition (in the form of Remark \ref{rmk:relative}) on closed points of $G$.
Let us normalise $\vartheta$, if necessary, to achieve $\vartheta_e = \id_E$. We need to show that the function
\[
F\colon G \times G \to \Aut E ,\quad (g,h) \mapsto \vartheta_{hg}^{-1} \circ \rho_g^\ast\vartheta_{h}\circ \vartheta_{g}
\]
is the constant $1 \in \BC^\ast \subset \Aut E$ where, as in \S\,\ref{sec:equiv_sheaves_preliminaries}, $\rho_g$ denotes the composition 
\[
X \,\simto\,\set{g} \times X \into G\times X\xrightarrow{\tau}X.
\]

We proceed as follows. Fix a closed point $m \in M$, sitting over $b \in B$. Set $E_m = \LL i_m^\ast E$, a perfect complex on $Y_b\simto Y_b \times \set{m} \into X$. Consider the commutative diagram
\begin{equation}\label{diag:restriction_of_action}
\begin{tikzcd}[row sep=large]
Y_b \times \set{m}\arrow{r}{\sim}\arrow[hook]{d}{i_m} &
\set{g}\times Y_b \times \set{m} \arrow[hook]{d}{g\times i_m}\arrow[hook]{r}{} &
G \times (Y_b \times \set{m})\arrow[hook]{d}{\id_G\times i_m}
\rar[-to,
to path={
([yshift=0.5ex]\tikztostart.east) --
([yshift=0.5ex]\tikztotarget.west) \tikztonodes}]{\sigma_{Y,m}}
\rar[-to,
to path={
([yshift=-0.5ex]\tikztostart.east) --
([yshift=-0.5ex]\tikztotarget.west) \tikztonodes}][swap]{p_{2,m}}
& Y_b \times \set{m} \arrow[hook]{d}{i_m} \\
X \arrow{r}{\sim}
& \set{g}\times X \arrow[hook]{r}
& G\times X
\rar[-to,
to path={
([yshift=0.5ex]\tikztostart.east) --
([yshift=0.5ex]\tikztotarget.west) \tikztonodes}]{\tau}
\rar[-to,
to path={
([yshift=-0.5ex]\tikztostart.east) --
([yshift=-0.5ex]\tikztotarget.west) \tikztonodes}][swap]{p_{2}}
& X
\end{tikzcd}
\end{equation}
where $\sigma_{Y,m}$ is the restriction of the action $\sigma_Y$ to the fibre $Y_b \subset Y$, identified with $Y_b\times \set{m} $. Restricting the isomorphism $\vartheta$ in \eqref{equiv_str_E} to the slice $G\times (Y_b \times \set{m}) \subset G\times X$ we obtain an isomorphism
\[
\vartheta(m) = \vartheta\big|_{G\times Y_b \times \set{m}}\colon p_{2,m}^\ast E_m \,\simto\, \sigma_{Y,m}^\ast E_m,
\]
and restricting $\vartheta(m)$ further to $Y_b \times \set{m}\simto \set{g} \times Y_b \times \set{m}$ we obtain 
\[
\vartheta_g(m) \colon E_m \,\simto\, \rho_g(m)^\ast E_m,
\]
where $\rho_g(m)$ is, as ever, the composition
\[
\rho_g(m)\colon Y_b \times \set{m} \simto \set{g} \times Y_b \times \set{m} \into G\times Y_b \times \set{m} \xrightarrow{\sigma_{Y,m}} Y_b \times \set{m}.
\]
Now, for $g$ and $h$ varying in $G$, we consider the diagrams
\[
\begin{tikzcd}
\rho_{g}(m)^\ast \rho_{h}(m)^\ast E_m \arrow[equal]{d} & & \rho_g(m)^\ast E_m \arrow[swap]{ll}{\rho_g(m)^\ast \vartheta_h(m)} \\
\rho_{hg}(m)^\ast E_m & & E_m \arrow[swap]{ll}{\vartheta_{hg}(m)}\arrow[swap]{u}{\vartheta_g(m)} 
\end{tikzcd}
\]
and we ask whether these commute. As before, we translate this condition as follows. We consider the function
\[
F_m\colon G\times G \to \Aut E_m = \BC^\ast, \quad (g,h) \mapsto  \vartheta_{hg}(m)^{-1} \circ \rho_g(m)^\ast \vartheta_h(m)\circ \vartheta_g(m) .
\]
Thanks to the fact that $E_m$ is simple, which implies $\Aut E_m = \BC^\ast$, the function $F_m$ can be seen as a regular nowhere vanishing function in $\OO(G\times G)^\times$. Thus by Fact \ref{rmk:Rosenlicht_trick} (applied to $Z=Z'=G$, which is irreducible since it is smooth and connected) we can write $F_m(g,h) = F_{1,m}(g)\cdot F_{2,m}(h)$ for $F_{i,m} \in \OO(G)^\times$. Since $F_m(g,1)=1=F_m(1,h)$ for all closed points $g$, $h\in G$ (we have used $\rho_e(m) = \id$ as well as the normalisation $\vartheta_e(m) = \id_{E_m}$), it follows that $F_m$ is constantly equal to $1 \in \BC^\ast$.
To finish the proof, it is enough to observe that by Diagram \eqref{diag:restriction_of_action} we have $\vartheta_g(m) =  i_m^\ast \vartheta_g$, in particular we have a commutative diagram
\[
\begin{tikzcd}
G\times G \arrow{r}{F} \arrow[swap]{dr}{1\,\equiv\, F_m} & \Aut E\arrow{d}{i_m^\ast} \\
& \Aut E_m
\end{tikzcd}
\]
and since $i_m^\ast$ is an isomorphism on $\BC^\times \subset \Aut E$ we conclude that $F \equiv 1$, as required.
\end{proof}

\subsection{The Atiyah class of the universal complex}
From now on, we 
endow $E \in \Perf X$ with the $G$-equivariant structure produced in \eqref{equiv_str_E}. We endow $M$ with the $G$-action $g\cdot m = \sigma_M(g^{-1},m)$, so that the projection $\pi_M\colon X \to M$ is $G$-equivariant.
We also assume that $M$ admits a $G$-equivariant embedding inside a smooth scheme.\footnote{Recall from Example \ref{rmk:G-equiv_emb} that this is easily achieved in many practical applications, where the moduli space is actually quasi-projective. However, this assumption can be removed, just as in \cite[\S\,4.5]{HT}. We leave the details to the reader.} By Lemma \ref{lemma:equivariant_embedding_implies_dagger} it follows that $M$ (and hence $X$, by \cite[Ex.~2.1.2\,(h)]{MR1106918}) satisfy $(\dagger)$. Finally, notice that since the  projection $\pi_M\colon X \to M$ is a proper (in fact, smooth and projective) morphism of $G$-schemes satisfying condition $(\dagger)$, equivariant Grothendieck duality applies to $\pi_M$. 

Our goal is to prove the following result, which is Theorem \ref{thm:main2} from the Introduction.

\begin{theorem}\label{main2_body}
Let $M\to B$ be as in \S\,\ref{sec:moduli_M}, and assume $G$ is reductive. Then the relative obstruction theory on $M\to B$ is naturally $G$-equivariant.
\end{theorem}

We recalled in the Introduction how the relative obstruction theory is obtained via the Atiyah class of the universal complex. We review this below, directly in the equivariant setting.

The complex $E$ has a well-defined truncated Atiyah class \eqref{def:truncated_atiyah_class}
\[
\At_E \,\in \,\Ext^1_X(E,E\otimes \BL_{X}),
\]
and our first task is now to prove its equivariance. The proof of Theorem \ref{main2_body} will follow almost immediately by equivariant Grothendieck duality.
In fact, the equivariance of the Atiyah class is now an easy corollary of Proposition \ref{prop:At_E_is_equiv} and the main result of the paper.

\begin{corollary}\label{cor:absolute_AT}
The Atiyah class $\At_E$ is naturally $G$-equivariant.
\end{corollary}

\begin{proof}
Follows by combining Proposition \ref{prop:At_E_is_equiv} with Theorem \ref{main_thm_BODY}.
\end{proof}

The $G$-equivariant maps $\pi_Y\colon X \to Y$ and $Y \to \Spec \BC$ induce a morphism $j\colon\BL_X \to \BL_{X/Y}$ in $\derived(\QCoh_X^G)$ by Lemma \ref{lemma:equivariant_transitivity_triangle}.
Composing $\At_E$ with $\id_E \otimes j[1]$ gives the relative Atiyah class
\[
\At_{E/Y} \colon E \to E\otimes \BL_{X}[1] \to E\otimes \BL_{X/Y}[1].
\]
Therefore Corollary \ref{cor:absolute_AT} immediately implies the following.
 

\begin{corollary}\label{cor:truncated_AT_is_equiv}
The relative Atiyah class $\At_{E/Y}$ is naturally $G$-equivariant.
\end{corollary}

By taking the image of $\id_E \in \Hom_X(E,E)^G$ under the composition $(\id_E\otimes j[1])_\ast \circ \At_{E\ast}$ we can view the relative truncated Atiyah class as an element
\[
\At_{E/Y} \,\in \,\Ext_X^1(E,E\otimes \BL_{X/Y})^G = \Ext_X^1(E,E\otimes \pi_M^\ast \BL_{M/B})^G,
\]
where we have observed 
that $\BL_{X/Y} = \pi_M^\ast \BL_{M/B}$ by \cite[\href{https://stacks.math.columbia.edu/tag/09DJ}{Tag 09DJ}]{stacks-project}.

\subsection{Proof of Theorem \ref{thm:main2}}\label{sec:thmB_proof}

From now on, we assume $G$ to be \emph{reductive} (cf.~\S\,\ref{subsec:reductives}). 
We shall exploit the splitting
\begin{equation}\label{splittings}
\RRlHom_X(E,E) = \OO_X \oplus \RRlHom_X(E,E)_0, 
\end{equation}
which we wish to prove to be $G$-equivariant.
Recall (see e.g.~\cite[\S\,10.1]{modulisheaves} for more details on this construction) how \eqref{splittings} is obtained in the non-equivariant setup: 
the trace map $\tr\colon \RRlHom_X(E,E) \to \OO_X$ splits the identity homomorphism $\id_E\colon \OO_X \to \RRlHom_X(E,E)$, and the composition $\tr\circ\id_E$ is multiplication by the rank $r$ (which we assumed nonzero in \S\,\ref{sec:moduli_M}). We now show that the induced distinguished triangle
\begin{equation}\label{triangle_traceless}
\RRlHom_X(E,E)_0 \to \RRlHom_X(E,E) \xrightarrow{\tr} \OO_X,
\end{equation}
defining the \emph{traceless} $\RRlHom$, is naturally lifted to $\derived(\QCoh_X^G)$.

Consider the element 
\begin{align*}
\id_E\,\in\,\Hom_X(E,E) = \Hom_X(\OO_X,\RRlHom_X(E,E)).
\end{align*}
By Remark \ref{identity_is_equivariant} and Lemma \ref{lemma:G-inv_Ext_groups}, we know that 
\[
\id_E\,\in\,\Hom_X(\OO_X,\RRlHom_X(E,E))^G \,\cong\,\Hom_{\derived(\QCoh_X^G)}(\OO_X,\RRlHom_X(E,E)).
\]
Then take $\mathscr F^\bullet = \OO_X$ and $i = \id_{E}$ in Corollary \ref{lemma:equivariant_composition}\,\eqref{item84762874} to observe that 
\[
\id_{E}^\ast\colon \Hom_X(\RRlHom_X(E,E),\OO_X) \to \Hom_X(\OO_X,\OO_X)
\]
is $G$-equivariant. In particular, it preserves the $G$-invariant parts. Since the trace map $\tr\in \Hom_X(\RRlHom_X(E,E),\OO_X)$ gets sent to $r\cdot \id_{\OO_X}$, which is $G$-invariant in virtue of Remark \ref{identity_is_equivariant}, it follows that $\tr$ must be $G$-invariant, too.
In other words,
\[
\tr\,\in\, \Hom_X(\RRlHom_X(E,E),\OO_X)^G \,\cong\,\Hom_{\derived(\QCoh_X^G)}(\RRlHom_X(E,E),\OO_X).
\]
We can then take the shifted cone of $\tr$ in $\derived(\QCoh_X^G)$ to obtain a distinguished triangle
\[
\RRlHom_X(E,E)_0 \to \RRlHom_X(E,E) \xrightarrow{\tr} \OO_X
\]
in $\Perf^G(X)\subset \derived(\QCoh_X^G)$, lifting \eqref{triangle_traceless}.

We have proved that the splitting \eqref{splittings} is $G$-equivariant. This in particular implies that the projection
\[
\mathsf q \colon \Ext^1_X(E,E\otimes \pi_M^\ast \BL_{M/B}) \to \Ext^1_X(\RRlHom_X(E,E)_0, \pi_M^\ast \BL_{M/B})
\]
from the full Ext group containing the element $\At_{E/Y}$, is a morphism of $G$-representations, in particular it 
preserves $G$-invariant parts. Therefore, $\At_{E/Y}$ maps to an element
\[
\mathsf q(\At_{E/Y})\,\in\, \Ext^1_X(\RRlHom_X(E,E)_0, \pi_M^\ast \BL_{M/B})^G.
\]
Note that the dualising line bundle $\omega_{\pi_M} = \pi_Y^\ast \omega_{Y/B}$ is naturally $G$-equivariant. By equivariant Grothendieck duality along the proper morphism $\pi_M$ 
(Corollary \ref{cor:Equivariant_Grothendieck_Global_sections}), 
the latter group is canonically isomorphic to
\begin{align*}
\Ext_X^1(\RRlHom_X(E,&E)_0\otimes \omega_{\pi_M}[d],\pi_M^\ast \BL_{M/B} \otimes \omega_{\pi_M} [d])^G 
\\
&=\,\,\Ext_X^{1-d}(\RRlHom_X(E,E)_0\otimes \omega_{\pi_M},\pi_M^!\BL_{M/B})^G 
\\
&\cong\,\, \Ext_M^{1-d}(\RR \pi_{M\ast}(\RRlHom_X(E,E)_0 \otimes \omega_{\pi_M}),\BL_{M/B})^G \\
&=\,\, \Hom_M(\mathbb E,\BL_{M/B})^G\\
&\cong\,\, \Hom_{\derived(\QCoh_M^G)}(\mathbb E,\BL_{M/B})
\end{align*}
where $d$ is the relative dimension of $Y\to B$ and we have set
\[
\mathbb E = \RR \pi_{M\ast}(\RRlHom_X(E,E)_0 \otimes \omega_{\pi_M})[d-1].
\]
We have also used again that $G$ is reductive for the last isomorphism. The morphism $\phi \in \Hom_M(\mathbb E,\BL_{M/B})$ determined as the image of the relative truncated Atiyah class $\At_{E/Y}$ is a relative obstruction theory on $M\to B$ by \cite[Thm.~4.1]{HT}. Therefore we have shown its equivariance in the sense of Definition \ref{def:equiv_OB_THEORY}.

The proof of Theorem \ref{thm:main2} is complete.

\begin{example}
Let $Y$ be a smooth complex projective toric $3$-fold. Let $G = \BG_m^3\subset Y$ be the open torus. The above result confirms the $G$-equivariance of the (perfect) obstruction theory on the following classical moduli spaces:
\begin{enumerate}
\item the Hilbert scheme of points $\Hilb^nY$,
\item the moduli space $I_m(Y,\beta)$ of ideal sheaves $\mathscr I$ with $\ch\, \mathscr I = (1,0,-\beta,-m)$,\label{item222}
\item the moduli space $P_m(Y,\beta)$ of stable pairs $(\mathcal F,s)$ with $\chi(\mathcal F)=m$ and $[\mathcal F] = \beta$,\label{item333}
\item higher rank analogues of \eqref{item222} and \eqref{item333},
\item the Quot scheme $\Quot_Y(F,n)$ of length $n$ quotients of a $G$-equivariant exceptional locally free sheaf $F$, as in \cite{Quot19}. This will be exploited in \cite{FMR}.
\end{enumerate}
\end{example}

\bibliographystyle{amsplain}
\bibliography{bib}

\end{document}

%% file: macros.tex
\usepackage[colorinlistoftodos]{todonotes}

\definecolor{antiquewhite}{rgb}{0.98, 0.92, 0.84}
\definecolor{buff}{rgb}{0.94, 0.86, 0.51}
\definecolor{palecopper}{rgb}{0.85, 0.54, 0.4}
\definecolor{fluorescentyellow}{rgb}{0.8, 1.0, 0.0}

\definecolor{britishracinggreen}{rgb}{0.0, 0.26, 0.15}
\definecolor{cobalt}{rgb}{0.0, 0.28, 0.67}
\DeclareSymbolFont{usualmathcal}{OMS}{cmsy}{m}{n}
\DeclareSymbolFontAlphabet{\mathcal}{usualmathcal}

\newcommand{\BC}{{\mathbb{C}}}

\newcommand{\BG}{{\mathbb{G}}}

\newcommand{\BL}{{\mathbb{L}}}

\newcommand{\BZ}{{\mathbb{Z}}}

\newcommand{\simto}{\,\widetilde{\to}\,}

\newcommand{\qcoh}{\mathsf{qc}}

\newcommand{\pt}{{\mathsf{pt}}}
\newcommand{\ch}{{\mathrm{ch}}}

\DeclareMathOperator{\Aut}{Aut}
\DeclareMathOperator{\Hilb}{Hilb}

\DeclareMathOperator{\Quot}{Quot}

\DeclareMathOperator{\Mod}{Mod}

\DeclareMathOperator{\Coh}{Coh}
\DeclareMathOperator{\QCoh}{QCoh}

\DeclareMathOperator{\At}{At}

\DeclareMathOperator{\Perf}{Perf}

\DeclareMathOperator{\Rep}{Rep}

\DeclareMathOperator{\op}{op}
\DeclareMathOperator{\tr}{tr}

\newcommand{\derived}{\mathbf{D}}
\newcommand{\dd}{\mathrm{d}}

\newcommand*{\defeq}{\mathrel{\vcenter{\baselineskip0.5ex \lineskiplimit0pt
                     \hbox{\scriptsize.}\hbox{\scriptsize.}}}%
                     =}

\newcommand{\into}{\hookrightarrow}
\newcommand{\onto}{\twoheadrightarrow}

\DeclareFontFamily{OT1}{rsfs}{}
\DeclareFontShape{OT1}{rsfs}{n}{it}{<-> rsfs10}{}
\DeclareMathAlphabet{\curly}{OT1}{rsfs}{n}{it}

\newcommand\Ext{\operatorname{Ext}}
\newcommand\Hom{\operatorname{Hom}}

\newcommand\LL{\mathbf{L}}
\newcommand\RR{\mathbf{R}}
\DeclareMathOperator{\lHom}{{\mathscr Hom}}
\DeclareMathOperator{\RRlHom}{{\mathbf{R}\mathscr Hom}}
\DeclareMathOperator{\RHom}{{\mathbf{R}\mathrm{Hom}}}

\newcommand\Spec{\operatorname{Spec}}

\newcommand\id{\operatorname{id}}

\newcommand{\Pic}{\mathop{\rm Pic}\nolimits}

\newcommand{\HH}{\mathrm{H}}
\newcommand{\OO}{\mathscr O}

\usepackage[all]{xy}
\usepackage{tikz}
\usepackage{tikz-cd}
\usepackage{rotating}
\usepackage{comment}
\newcommand*{\isoarrow}[1]{\arrow[#1,"\rotatebox{90}{\(\sim\)}"
]}
\usetikzlibrary{matrix,shapes,intersections,arrows,decorations.pathmorphing}
\tikzset{commutative diagrams/arrow style=math font}
\tikzset{commutative diagrams/.cd,
mysymbol/.style={start anchor=center,end anchor=center,draw=none}}
\newcommand\MySymb[2][\square]{%
  \arrow[mysymbol]{#2}[description]{#1}}
\tikzset{
shift up/.style={
to path={([yshift=#1]\tikztostart.east) -- ([yshift=#1]\tikztotarget.west) \tikztonodes}
}
}

\theoremstyle{definition}

\newtheorem*{lemma*}{Lemma}
\newtheorem*{theorem*}{Theorem}
\newtheorem*{example*}{Example}
\newtheorem*{fact*}{Fact}
\newtheorem*{notation*}{Notation}
\newtheorem*{definition*}{Definition}
\newtheorem*{prop*}{Proposition}
\newtheorem*{remark*}{Remark}
\newtheorem*{corollary*}{Corollary}

\newtheorem*{conventions*}{Conventions}

\newtheorem{definition}{Definition}[section]

\newtheorem{example}[definition]{Example}
\newtheorem{fact}[definition]{Fact}

\newtheorem{notation}[definition]{Notation}
\newtheorem{remark}[definition]{Remark}

\newtheoremstyle{thm} 
        {3mm}
        {3mm}
        {\slshape}
        {0mm}
        {\bfseries}
        {.}
        {1mm}
        {}
\theoremstyle{thm}
\newtheorem{theorem}[definition]{Theorem}
\newtheorem{corollary}[definition]{Corollary}
\newtheorem{lemma}[definition]{Lemma}
\newtheorem{prop}[definition]{Proposition}
\newtheorem{thm}{Theorem}

\newtheoremstyle{ex} 
        {3mm}
        {3mm}
        {}
        {0mm}
        {\scshape}
        {.}
        {1mm}
        {}
\theoremstyle{ex}

\newtheoremstyle{sol} 
        {3mm}
        {3mm}
        {}
        {0mm}
        {\scshape}
        {.}
        {1mm}
        {}
\theoremstyle{sol}

%% file: main.bbl
\providecommand{\bysame}{\leavevmode\hbox to3em{\hrulefill}\thinspace}
\providecommand{\MR}{\relax\ifhmode\unskip\space\fi MR }
\providecommand{\MRhref}[2]{%
  \href{http://www.ams.org/mathscinet-getitem?mr=#1}{#2}
}
\providecommand{\href}[2]{#2}
\begin{thebibliography}{10}

\bibitem{MR86359}
Michael~Francis Atiyah, \emph{Complex analytic connections in fibre bundles},
  Trans. Amer. Math. Soc. \textbf{85} (1957), 181--207.

\bibitem{BFK}
Matthew Ballard, David Favero, and Ludmil Katzarkov, \emph{Variation of
  geometric invariant theory quotients and derived categories}, J. Reine Angew.
  Math. \textbf{746} (2019), 235--303.

\bibitem{BR18}
Sjoerd Beentjes and Andrea~T. Ricolfi, \emph{{Virtual counts on Quot schemes
  and the higher rank local DT/PT correspondence}},
  {\url{https://arxiv.org/abs/1811.09859}}, To appear in Math.~Res.~Lett.,
  2018.

\bibitem{BFinc}
Kai Behrend and Barbara Fantechi, \emph{The intrinsic normal cone}, Inventiones
  Mathematicae \textbf{128} (1997), no.~1, 45--88.

\bibitem{BFHilb}
\bysame, \emph{{\itshape{Symmetric obstruction theories and Hilbert schemes of
  points on threefolds}}}, Algebra Number Theory \textbf{2} (2008), 313--345.

\bibitem{SGA6}
P.~{Berthelot}, A.~{Grothendieck}, and L.~{Illusie} (eds.), \emph{{S\'eminaire
  de g\'eom\'etrie alg\'ebrique du Bois Marie 1966/67, SGA 6.Dirig\'e par P.
  Berthelot, A. Grothendieck et L. Illusie, Avec la collaboration de D.
  Ferrand, J. P. Jouanolou, O. Jussilia, S. Kleiman, M. Raynaud et J. P. Serre.
  Th\'eorie des intersections et th\'eor\`eme de Riemann-Roch.}}, vol. 225,
  Springer, Cham, 1971 (French).

\bibitem{DavisonR}
Ben Davison and Andrea~T. Ricolfi, \emph{{The local motivic DT/PT
  correspondence}}, {\url{https://arxiv.org/abs/1905.12458}}, 2019.

\bibitem{Dolgachev_GIT}
Igor Dolgachev, \emph{Lectures on invariant theory}, London Mathematical
  Society Lecture Note Series, vol. 296, Cambridge University Press, Cambridge,
  2003.

\bibitem{Fantechi_Gottsche}
Barbara Fantechi and Lothar G\"{o}ttsche, \emph{Riemann-{R}och theorems and
  elliptic genus for virtually smooth schemes}, Geom. Topol. \textbf{14}
  (2010), no.~1, 83--115.

\bibitem{FMR}
Nadir Fasola, Sergej Monavari, and Andrea~T. Ricolfi, \emph{{Higher rank
  K-theoretic Donaldson--Thomas Theory of points}}, In preparation.

\bibitem{Gholampour2017a}
Amin Gholampour, Martijn Kool, and Benjamin Young, \emph{{Rank 2 Sheaves on
  Toric 3-Folds: Classical and Virtual Counts}}, Int. Math. Res. Not.
  \textbf{2018} (2018), no.~10, 2981--3069.

\bibitem{GPvirtual}
Tom Graber and Rahul Pandharipande, \emph{Localization of virtual classes},
  Invent. Math. \textbf{135} (1999), no.~2, 487--518.

\bibitem{MR0217083}
Alexander~Grothendieck, \emph{{\'E}l{\'e}ments de g{\'e}om{\'e}trie alg{\'e}brique.
  {I}. {L}e langage des sch{\'e}mas}, Inst. Hautes {\'E}tudes Sci. Publ. Math.
  (1960), no.~4, 228. 

\bibitem{Hall_Neeman_Rydh}
Jack Hall, Amnon Neeman, and David Rydh, \emph{One positive and two negative
  results for derived categories of algebraic stacks}, J. Inst. Math. Jussieu
  \textbf{18} (2019), no.~5, 1087--1111.

\bibitem{MR3436239}
Jack Hall and David Rydh, \emph{Algebraic groups and compact generation of
  their derived categories of representations}, Indiana Univ. Math. J.
  \textbf{64} (2015), no.~6, 1903--1923. 

\bibitem{Hall_Rydh_1}
\bysame, \emph{Perfect complexes on algebraic stacks}, Compos. Math.
  \textbf{153} (2017), no.~11, 2318--2367.

\bibitem{modulisheaves}
Daniel Huybrechts and Manfred Lehn, \emph{The geometry of moduli spaces of
  sheaves}, second ed., Cambridge Mathematical Library, Cambridge University
  Press, Cambridge, 2010.

\bibitem{HT}
Daniel Huybrechts and Richard~P. Thomas, \emph{Deformation-obstruction theory
  for complexes via {A}tiyah and {K}odaira-{S}pencer classes}, Math. Ann.
  \textbf{346} (2010), no.~3, 545--569.

\bibitem{HT_erratum}
\bysame, \emph{Erratum to: {D}eformation-obstruction theory for complexes via
  {A}tiyah and {K}odaira-{S}pencer classes [mr2578562]}, Math. Ann.
  \textbf{358} (2014), no.~1-2, 561--563.

\bibitem{Ill}
Luc Illusie, \emph{Complexe cotangent et d\'{e}formations. {I}}, Lecture Notes
  in Mathematics, Vol. 239, Springer-Verlag, Berlin-New York, 1971.

\bibitem{IllusieII}
\bysame, \emph{Complexe cotangent et d\'{e}formations. {II}}, Lecture Notes in
  Mathematics, Vol. 283, Springer-Verlag, Berlin-New York, 1972.

\bibitem{MR3334894}
Srikanth~B. Iyengar, Joseph Lipman, and Amnon Neeman, \emph{Relation between
  two twisted inverse image pseudofunctors in duality theory}, Compos. Math.
  \textbf{151} (2015), no.~4, 735--764.

\bibitem{Kool_Fixed_Point_Loci}
Martijn Kool, \emph{Fixed point loci of moduli spaces of sheaves on toric
  varieties}, Adv. Math. \textbf{227} (2011), no.~4, 1700--1755.

\bibitem{LMB}
G\'{e}rard Laumon and Laurent Moret-Bailly, \emph{Champs alg\'{e}briques},
  Ergebnisse der Mathematik und ihrer Grenzgebiete. 3. Folge. A Series of
  Modern Surveys in Mathematics [Results in Mathematics and Related Areas. 3rd
  Series. A Series of Modern Surveys in Mathematics], vol.~39, Springer-Verlag,
  Berlin, 2000.

\bibitem{LiTian}
Jun Li and Gang Tian, \emph{Virtual moduli cycles and {G}romov-{W}itten
  invariants of algebraic varieties}, J. Amer. Math. Soc. \textbf{11} (1998),
  no.~1, 119--174.

\bibitem{Lipman1}
Joseph Lipman, \emph{Notes on derived functors and {G}rothendieck duality},
  Foundations of {G}rothendieck duality for diagrams of schemes, Lecture Notes
  in Math., vol. 1960, Springer, Berlin, 2009, pp.~1--259.

\bibitem{MFK}
David Mumford, John Fogarty, and Frances~C. Kirwan, \emph{Geometric invariant
  theory}, third ed., Ergebnisse der Mathematik und ihrer Grenzgebiete (2)
  [Results in Mathematics and Related Areas (2)], vol.~34, Springer-Verlag,
  Berlin, 1994.

\bibitem{Neeman2}
Amnon Neeman, \emph{The {G}rothendieck duality theorem via {B}ousfield's
  techniques and {B}rown representability}, J. Amer. Math. Soc. \textbf{9}
  (1996), no.~1, 205--236.

\bibitem{Neeman0}
\bysame, \emph{{An improvement of the base-change theorem and the functor
  $f^!$}}, {\url{https://arxiv.org/abs/1406.7599}}, 2017.

\bibitem{Olsson}
Martin Olsson, \emph{Sheaves on {A}rtin stacks}, J. Reine Angew. Math.
  \textbf{603} (2007), 55--112.

\bibitem{Olsson_book}
\bysame, \emph{Algebraic spaces and stacks}, American Mathematical Society
  Colloquium Publications, vol.~62, American Mathematical Society, Providence,
  RI, 2016.

\bibitem{Oprea:2019aa}
Dragos Oprea and Rahul Pandharipande, \emph{Quot schemes of curves and
  surfaces: virtual classes, integrals, euler characteristics}, 2019.

\bibitem{Ricolfi2018}
Andrea~T. Ricolfi, \emph{The {DT}/{PT} correspondence for smooth curves}, Math.
  Z. \textbf{290} (2018), no.~1-2, 699--710.

\bibitem{LocalDT}
\bysame, \emph{{Local contributions to Donaldson--Thomas invariants}},
  International Mathematics Research Notices \textbf{2018} (2018), no.~19,
  5995--6025.

\bibitem{Quot19}
\bysame, \emph{{Virtual classes and virtual motives of Quot schemes on
  threefolds}}, {\url{https://arxiv.org/abs/1906.02557}}, 2019.

\bibitem{Rosenlicht}
Maxwell Rosenlicht, \emph{Toroidal algebraic groups}, Proc. Amer. Math. Soc.
  \textbf{12} (1961), 984--988.

\bibitem{MR1948842}
C.~Serp\'{e}, \emph{Resolution of unbounded complexes in {G}rothendieck
  categories}, J. Pure Appl. Algebra \textbf{177} (2003), no.~1, 103--112.

\bibitem{Spaltenstein}
N.~Spaltenstein, \emph{Resolutions of unbounded complexes}, Compositio Math.
  \textbf{65} (1988), no.~2, 121--154.

\bibitem{stacks-project}
The {Stacks Project Authors}, \emph{{\itshape{Stacks Project}}},
  \url{http://stacks.math.columbia.edu}, 2016.

\bibitem{Sumihiro1975}
Hideyasu Sumihiro, \emph{Equivariant completion ii}, J. Math. Kyoto Univ.
  \textbf{15} (1975), no.~3, 573--605.

\bibitem{Thomason}
Robert~W. Thomason, \emph{Equivariant resolution, linearization, and
  {H}ilbert's fourteenth problem over arbitrary base schemes}, Adv. in Math.
  \textbf{65} (1987), no.~1, 16--34.

\bibitem{MR1106918}
Robert~W. Thomason and Thomas Trobaugh, \emph{Higher algebraic {$K$}-theory of
  schemes and of derived categories}, The {G}rothendieck {F}estschrift, {V}ol.
  {III}, Progr. Math., vol.~88, Birkh\"{a}user Boston, Boston, MA, 1990,
  pp.~247--435.

\bibitem{Totaro}
Burt Totaro, \emph{The resolution property for schemes and stacks}, J. Reine
  Angew. Math. \textbf{577} (2004), 1--22.

\bibitem{VV}
Michela Varagnolo and Eric Vasserot, \emph{Double affine hecke algebras and
  affine flag manifolds, i}, Affine Flag Manifolds and Principal Bundles
  (Basel) (Alexander Schmitt, ed.), Springer Basel, 2010, pp.~233--289.

\end{thebibliography}
